\documentclass[12pt]{article}

\usepackage[english]{babel}
\usepackage{enumerate}
\usepackage{stix}
\usepackage{xcolor}
\usepackage{authblk}

\usepackage{amsmath}
\usepackage{amsfonts}
\usepackage{amsthm}
\usepackage{mathtools}
\usepackage{graphicx}
\usepackage[colorinlistoftodos]{todonotes}
\usepackage[colorlinks=true, allcolors=blue]{hyperref}

\usepackage{tikz}

\theoremstyle{plain}
\newtheorem{theorem}{Theorem}[section]
\newtheorem{proposition}[theorem]{Proposition}
\newtheorem{corollary}[theorem]{Corollary}
\newtheorem{lemma}[theorem]{Lemma}

\theoremstyle{definition}
\newtheorem{definition}[theorem]{Definition}

\newtheorem{example}[theorem]{Example}

\newcommand{\R}{\mathbb{R}}

\newcommand{\Z}{\mathbb{Z}}
\newcommand{\C}{\mathbb{C}}
\newcommand{\Q}{\mathbb{Q}}

\newcommand{\1}{\mathbf 1}

\newcommand{\ii}{\mathrm{i}}

\newcommand{\0}{\mathbf 0}
\newcommand{\cj}[1]{\overline{#1}}
\newcommand{\wt}[1]{\widetilde{#1}}
\newcommand{\bttm}[1]{\widecheck{#1}}
\newcommand{\tp}[1]{\widehat{#1}}
\newcommand{\md}[1]{\ (\operatorname{mod}   #1)}
\newcommand{\cart}{%
  \text{\fboxsep=0.5pt\fbox{\rule[1pt]{0pt}{1ex}\rule[1pt]{1ex}{0pt}}}%
}

\DeclareMathOperator{\spn}{span}
\DeclareMathOperator{\rk}{rank}

\DeclarePairedDelimiter{\alg}{\langle}{\rangle}

\newcommand{\Hidden}[1]{}

\title{Fundamentals of fractional revival in graphs}
\author[1]{Ada Chan\thanks{Corresponding author: ssachan@yorku.ca}}
\author[2]{Gabriel Coutinho}
\author[3]{Whitney Drazen}
\author[4]{Or Eisenberg}
\author[5]{Chris Godsil}
\author[3]{Gabor Lippner}
\author[6]{Mark Kempton}
\author[7]{Christino Tamon}
\author[1]{Hanmeng Zhan}
\affil[1]{\small{Department of Mathematics and Statistics, York University}} 
\affil[2]{Department of Computer Science, Universidade Federal de Minas Gerais}
\affil[3]{Department of Mathematics, Northeastern University}
\affil[4]{Department of Mathematics, Harvard University}
\affil[5]{Department of Combinatorics and Optimization, University of Waterloo}
\affil[6]{Department of Mathematics, Brigham Young University}
\affil[7]{Department of Computer Science, Clarkson University}

\date{\today}

\begin{document}
\maketitle

\begin{abstract}
We develop a general spectral framework to analyze quantum fractional revival in quantum spin networks.  In particular, we introduce generalizations of the notions of cospectral and strongly cospectral vertices to arbitrary subsets of vertices, and give various examples. This work resolves two open questions of Chan et.~al. [``Quantum Fractional Revival on graphs". \textit{Discrete Applied Math}, 269:86-98, 2019.]
\end{abstract}

\section{Introduction}

An important problem in quantum information theory is the transfer of a quantum state through a quantum spin network.  The use of a quantum walk on a graph to analyze such information transfer was initiated by Bose \cite{BoseQuantumComPaths}, and carried further by Christandl et.~al., among others \cite{ChristandlPSTQuantumSpinNet2,ChristandlPSTQuantumSpinNet}.  Comprehensive surveys have been given by Kay \cite{KayReviewPST,KayPerfectcommunquantumnetworks}.  Many tools from algebraic graph theory have found applications to this problem, and a survey of work in this area is given by Godsil \cite{GodsilStateTransfer12}.

We study the single-excitation subspace of a spin network with uniform $\texttt{XX}$ couplings.  We denote the network by $X$, the set of vertices by $V(X)$ and the set of edges by $E(X)$. The evolution of such a system is given by its Hamiltonian
\[ H_{\texttt{XX}} = \frac{1}{2}\sum_{ (i,j) \in E(X)}a_{i,j}( \texttt{X}_i \texttt{X}_j + \texttt{Y}_i \texttt{Y}_j) + \sum_{i \in V(G)} a_{i,i} \cdot \texttt{Z}_i ,\] where $\texttt{X}_i, \texttt{Y}_i, \texttt{Z}_i$ are the standard Pauli matrices. The first sum corresponds to the  $\texttt{XX}$ couplings with coupling strength $a_{i,j}$ and $a_{i,i}$ represents the strength of a magnetic field at node $i$. The restriction of this system to the single-excitation subspace leads to the unitary transition matrix
\[
U(t):= e^{-itA}
\]
for $t\geq 0$, where $A=[a_{i,j}]$ is the (weighted) adjacency matrix of $X$.  The matrix $U(t)$ represents a continuous-time quantum walk on the graph $X$. 

The principal question that has been studied in this area is that of \emph{perfect state transfer}, namely when a quantum state can be transferred from node $a$ to $b$ with perfect fidelity.  This means
\[
|U(t)_{a,b}|=1.
\]

A generalization of perfect state transfer is a phenomenon called \emph{fractional revival} in a graph.  
Let $e_a$ denote the characteristic vector for vertex $a$.
For $K\subset V(X)$, we say $K$-fractional revival occurs in $X$ if there is some $t>0$ such that $U(t)e_a$ is supported only on $K$, for any $a\in K$.   Fractional revival in quantum spin networks has been studied in \cite{chen2007fractional,GenestVinetZhedanov1,GenestVinetZhedanov,christandl2017analytic,ChanCoutinhoTamonVinetZhan,bernard2018graph}.  Most of the literature on this topic concerns fractional revival between pairs of vertices ($|K|=2$).  

In this paper, we study $K$-fractional revival for subset $K$ of vertices of arbitrary size.  We develop a theoretical spectral framework for studying fractional revival in graphs.  We take many of the theoretical tools that have been developed in the study of perfect state transfer and generalize them appropriately to apply to fractional revival.  Specifically, a critical necessary condition for perfect state transfer between two vertices $a$ and $b$ is that they be \emph{cospectral}, that is, that $X\backslash a$ and $X\backslash b$ have the same (adjacency) spectrum.  In fact, a stronger condition, called \emph{strong cospectrality} is necessary for perfect state transfer \cite{GodsilStateTransfer12}.  As a generalization, we develop the notions of decomposability and strongly fractional cospectrality, and show that they are necessary conditions for fractional revival.  In addition, we explore conditions on eigenvalues that are necessary for fractional revival, analogous to eigenvalue conditions needed for perfect state transfer. 

Chan et.~al.~\cite{ChanCoutinhoTamonVinetZhan} studied fractional revival between pairs of cospectral vertices.  They posed as a significant open question what conditions are necessary for fractional revival to occur between non-cospectral pairs.  The spectral framework provided in this paper resolves this question.  In addition, they posed the question of whether or not it is possible for three or more vertices to exhibit pairwise fractional revival or not.  In other words, does fractional revival exhibit the ``monogamy'' property that is known (see \cite{KayPerfectcommunquantumnetworks}) to hold for perfect state transfer?  We resolve this question by constructing graphs in which  fractional revival occurs between every pair of vertices.  Thus fractional revival is not ``monogamous'' like perfect state transfer.  This demonstrates a fundamental difference between fractional revival and perfect state transfer.

\section{Preliminaries}
\label{Preliminaries}

In this paper, $X$ is a connected graph on $n$ vertices and $A$ is its adjacency matrix. We allow $X$ to have loops and  real weights on its edges. The continuous-time quantum walk of $X$ is given by the transition operator
\begin{equation*}
U(t) = e^{-\ii t A}.
\end{equation*}
Given the spectral decomposition
\begin{equation}
\label{Eqn:ASpec}
A = \sum_{r=0}^d \theta_r E_r,
\end{equation}
we can write
\begin{equation}
\label{Eqn:USpec}
U(t) =  \sum_{r=0}^d e^{-\ii t \theta_r} E_r.
\end{equation}
We call $E_r$, the orthogonal projection onto the $\theta_r$-eigenspace of $A$, a \textsl{principal idempotent} of $A$.

\begin{definition}
\label{Def:FR}
A graph $X$ has \textsl{fractional revival} from vertex $a$ to $b$ at time $\tau$ if
\begin{equation*}
U(\tau) e_a = \alpha e_a + \beta e_b,
\end{equation*}
for some $\alpha, \beta \in \C$ satisfying $|\alpha|^2+|\beta|^2=1$. If $\beta = 0$, we further require that $U(\tau) e_b = \gamma e_b$ for some $\gamma \in \C$ with $|\gamma| = 1$.
\end{definition}
If $\alpha=0$ then perfect state transfer occurs from $a$ to $b$ at time $\tau$. If $\beta=0$ then $X$ is periodic at $a$ and $b$ at time $\tau$. We say \textsl{proper} fractional revival occurs if $\alpha\beta \neq 0$.

Without loss of generality, we assume $a$ and $b$ index the first two columns and rows of $A$ and $U(t)$.
If there is fractional revival from $a$ to $b$ at time $\tau$, the first column of  $U(\tau)$ is $\begin{bmatrix} \alpha, \beta, 0,\ldots,0\end{bmatrix}^T$.  Since $U(\tau)$ is symmetric and unitary, 
\begin{equation*}
U(\tau) = 
\begin{bmatrix}
H & \0\\ \0 & H'
\end{bmatrix},
\end{equation*}
where 
\begin{equation*}
H=\begin{bmatrix}
\alpha & \beta \\ \beta & -\frac{\cj{\alpha}}{\cj{\beta}}\beta
\end{bmatrix} 
\end{equation*} if $\beta\neq0$ or \begin{equation*}
H=\begin{bmatrix}
\alpha & 0\\ 0 & \gamma
\end{bmatrix} 
\end{equation*} if $\beta=0$.

\begin{example}
\label{Ex:Cocktail}
The cocktail party graph \cite{ChanCoutinhoTamonVinetZhan2} has fractional revival between antipodal pairs of vertices at time $\tau = \frac{j\pi}{n}$, for $j=1,\ldots,n$, 
with
\begin{equation*}
U(\frac{j\pi}{n})=
I_n \otimes \begin{bmatrix} \cos(\frac{j(n-1)\pi}{n}) & \ii \sin(\frac{j(n-1)\pi}{n})\\  \ii \sin(\frac{j(n-1)\pi}{n}) &\cos(\frac{j(n-1)\pi}{n}) \end{bmatrix}.
\end{equation*}
\end{example}

Observe that when fractional revival occurs from $a$ to $b$,
the transition matrix has the block diagonal form 
\begin{equation*}
U(\tau) = 
\begin{bmatrix}
H & \0\\ \0 & H'
\end{bmatrix},
\end{equation*}
with $H$ being a unitary matrix indexed by $a$ and $b$.  
In this paper, we study a generalization, called $K$-fractional revival, for some subset $K$ of vertices of arbitrary size, where the transition matrix at fractional revival time is block diagonal with one block being indexed by the vertices in $K$.

The following example shows behavior pointing to possible generalization beyond the notion of $K$-fractional revival.
\begin{example}
\label{Ex:Spider}
Let $X_m$ be the graph obtained from subdividing every edge in $K_{1,m}$, for $m\geq 2$.  
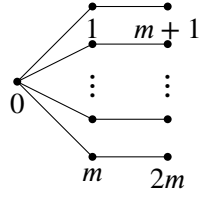
\begin{center}
\begin{tikzpicture}
\fill (0,0) circle (1.5pt);
\draw (0,0) node[anchor=north]{\small $0$};
\fill (1,1) circle (1.5pt);
\draw (1,1) node[anchor=north]{\small $1$};
\fill (1,0.5) circle (1.5pt);
\draw (1,0.25) node[anchor=north]{{$\vdots$}};
\fill (1,-0.5) circle (1.5pt);
\fill (1,-1) circle (1.5pt);
\draw (1,-1) node[anchor=north]{\small $m$};
\draw (0,0)--(1,1);
\draw (0,0)--(1,0.5);
\draw (0,0)--(1,-0.5);
\draw (0,0)--(1,-1);

\fill (2,1) circle (1.5pt);
\draw (2,1) node[anchor=north]{\small $m+1$};
\fill (2,0.5) circle (1.5pt);
\draw (2,0.25) node[anchor=north]{{$\vdots$}};
\fill (2,-0.5) circle (1.5pt);
\fill (2,-1) circle (1.5pt);
\draw (2,-1) node[anchor=north]{\small $2m$};
\draw (2,1)--(1,1);
\draw (2,0.5)--(1,0.5);
\draw (2,-0.5)--(1,-0.5);
\draw (2,-1)--(1,-1);

\end{tikzpicture}
\end{center}
Let $H_m = I_m - (1/m)J_m$. The transition matrix, $U(t)$, of $X_m$ is
\begin{equation*}
\begin{bmatrix} 
\frac{1+m\cos (t\sqrt{m+1})}{m+1} && \frac{-\ii\sin(t\sqrt{m+1})}{\sqrt{m+1}} \1^T_m&& \frac{-1 + \cos(t\sqrt{m+1})}{m+1} \1^T_m\\
\\
\frac{-\ii\sin(t\sqrt{m+1})}{\sqrt{m+1}}  \1_m &&\cos t  \Big(H_m\Big) +\frac{\cos(t\sqrt{m+1})}{m}J_m && -\ii\sin t \Big(H_m\Big) - \frac{\ii\sin(t\sqrt{m+1})}{m\sqrt{m+1}}J_m\\ 
\\
\frac{-1 + \cos(t\sqrt{m+1})}{m+1}  \1_m &&  -\ii\sin t \Big(H_m \Big) - \frac{\ii\sin(t\sqrt{m+1})}{m\sqrt{m+1}}J_m && \cos t \big(H_m\big) + \frac{m+\cos(t\sqrt{m+1})}{m(m+1)}J_m
 \end{bmatrix},
\end{equation*}
where $\1_m$ is the vector of all ones of length $m$ and $J_m$ is the $m\times m$ matrix of all ones.

At time $\tau=\frac{\pi}{\sqrt{m+1}}$, the support of 
$U(\tau)e_0$ consists of the vertex $0$ and all the leaves but the transition matrix
\begin{equation*}
U(\tau)=\begin{bmatrix} 
\frac{1-m}{1+m} & \0^T_m& \frac{-2 }{m+1} \1^T_m\\
\\
\0_m &\cos \tau  I_m -\frac{(1+\cos \tau)}{m}  J_m & -\ii\sin \tau \Big(I_m -\frac{1}{m} J_m\Big)\\ 
\\
 \frac{-2 }{m+1}  \1_m &  -\ii\sin \tau \Big(I_m -\frac{1}{m} J_m\Big) & \cos \tau I_m +\frac{m-1-(m+1)\cos \tau}{m(m+1)} J_m
 \end{bmatrix}
 \end{equation*}
is not block diagonal.  
\end{example}


\section{$K$-Fractional revival in graphs}
\label{Section:K-FR}

In this section, we generalize fractional revival between two vertices to fractional revival among all vertices of some arbitrary subset $K$ of vertices.  That is,
there exists a time $\tau$ such that $U(\tau)$ is a block diagonal matrix with one of the blocks being indexed by the elements in $K$.
Without loss of generality, we assume the top $|K|$ rows and the leftmost $|K|$ columns of $U(t)$ are indexed by the elements of $K$.

\begin{definition}
Let $K \subset V(X)$.  We say $X$ has $K$-fractional revival at time $\tau$ if, up to permuting rows and columns, the transition matrix $U(\tau)$ has the following block structure
\begin{equation*}
\begin{bmatrix} H & \0\\ \0 & H'\end{bmatrix},
\end{equation*}
for some $|K|\times |K|$ unitary matrix $H$.
\end{definition}
The definition includes the case where $H$ is a diagonal matrix, that is, every vertex in $K$ is periodic at time $\tau$.
We say $X$ has \textsl{proper} $K$-fractional revival if $H$ is not a diagonal matrix.
Proper $\{a,b\}$-fractional revival occurs in $X$ if and only if  it admits proper fractional revival or perfect state transfer between $a$ and $b$.

In \cite{GodsilRealState}, Godsil presents continuous-time quantum walk from the viewpoint of density matrices.
Density matrices are positive semidefinite matrices with trace one.    
Given the density matrix $D$ as the initial state, the state of the walk at time $t$ is 
\begin{equation*}
U(t)D U(-t).
\end{equation*}
We say the state $D$ is periodic at time $\tau$ if 
$U(\tau)D U(-\tau)=D$.

\begin{proposition}
Let $K\subset V(X)$ and $D_K =\sum_{a\in K} e_a e_a^{T}$.
Then $K$-fractional revival occurs in $X$ at time $\tau >0$ if and only if the density matrix $\frac{1}{|K|}D_K$ is periodic at $\tau$.
\end{proposition}
\begin{proof}
We have $K$-fractional revival in $X$ at time $\tau >0$ if and only if
\begin{equation*}
U(\tau)   D_K = D_K U(\tau),
\end{equation*}
equivalently,
\begin{equation*}
U(\tau)  \Big(\frac{1}{|K|} D_K \Big) U(-\tau) = \frac{1}{|K|}  D_K.
\end{equation*}
\end{proof}

It follows from Equation~(\ref{Eqn:USpec}) that $\frac{1}{|K|} D_K$ is periodic at time $\tau>0$ if and only if, for each $r$ and $s$,
\begin{equation*}
e^{i\tau(\theta_r-\theta_s)} E_r D_K E_s = E_r D_K E_s.
\end{equation*}
The \textsl{eigenvalue support} of the set $K$ is the set
\begin{equation*}
\Phi_K = \{(\theta_r, \theta_s): E_r D_K E_s\neq 0\}.
\end{equation*}
For $D_K$ to be periodic at time $\tau$, we must have
\begin{equation*}
e^{i\tau (\theta_r-\theta_s)}= 1, \quad \text{for $(\theta_r, \theta_s) \in \Phi_K$,}
\end{equation*}
which means that for any $(\theta_r,\theta_s), (\theta_h,\theta_k) \in \Phi_K$ with $\theta_h\neq \theta_k$, the ratio 
\begin{equation*}
\frac{\theta_r-\theta_s}{\theta_h-\theta_k}
\end{equation*}
must be rational.

Theorem~5.5 and Corollary~5.6 of \cite{GodsilRealState} give the following necessary condition of $K$-fractional revival
when the weights of $X$ are integers,
\begin{theorem}
\label{Thm:K-FReigenvalues}
Suppose $X$ has integer weights.  
If $K$-fractional revival occurs in $X$ at time $\tau>0$ then there exists a square-free integer $\Delta$ such that
$\theta_r-\theta_s$ is an integer multiple of $\sqrt{\Delta}$, for all $(\theta_r, \theta_s)\in \Phi_K$.
Moreover, the first instance of $K$-fractional revival occurs at time  at most $2\pi$.
\end{theorem}

\section{Constructions}
\label{Section:Constructions}

Graph products, whose adjacency matrices  have  natural block structures, are promising places to look for examples of $K$-fractional revival.
In this section, we give examples of $K$-fractional revival using some common graph products.  
We use $A_X$ and $U_X(t)$ to denote the adjacency matrix and the transition operator of the quantum walk for the graph $X$, respectively.

\subsection{Cartesian product}
Given graphs $X$ and $Y$, we use $X \cart Y$ to denote their Cartesian product.   Then
\begin{equation*}
U_{X \cart Y}(t) = U_X(t) \otimes U_Y(t).
\end{equation*}
In general, if $X$ is periodic at a vertex $a$ at time $\tau$ and $U_Y(\tau)$ is not a diagonal matrix, then $X\cart Y$ has proper $K$-fractional revival at time $\tau$
where $K =\{a\} \times V(Y)$.
Please see Example~3.3 of \cite{ChanCoutinhoTamonVinetZhan} for an example of this construction.

\subsection{Direct product}
The direct product of $X$ and $Y$, denoted by $X \times Y$, has adjacency matrix $A_X \otimes A_Y$.  By Lemma~4.2 of \cite{CoutinhoGodsilPSTProducts},
given the spectral decomposition $A_X=\sum_{r=0}^d \theta_r E_r$,
\begin{equation*}
U_{X\times Y} (t)  = \sum_{r=0}^d E_r \otimes U_Y(\theta_r t).
\end{equation*}
If there exists a non-diagonal matrix $H$ such that $U_Y(\theta_r \tau)=H$, for every $\theta_r$ in the spectrum of $X$, then
\begin{equation*}
U_{X\times Y} (\tau)  = \sum_{r=0}^d E_r \otimes U_Y(\theta_r \tau) = I_{|X|} \otimes H.
\end{equation*}
In this case, $X\times Y$ has proper $\big(\{a\} \times V(Y)\big)$-fractional revival at time $\tau$, for any $a \in V(X)$.
\begin{example}
\label{Ex:Direct}
The line graph of the complete bipartite graph $K_{n,n}$ has spectrum $\{2(n-1),n-2,-2\}$, see \cite{BrouwerHaemers}.
Let $X$ be the line graph of the complete bipartite graph $K_{16m,16m}$.
Let $Y$ be the $d$-cube, which is the Cartesian product of $d$ copies of $K_2$, and
\begin{equation*}
U_Y(t) = U_{K_2}(t)^{\otimes d} = \begin{bmatrix} \cos t & -\ii\sin t\\-\ii\sin t& \cos t\end{bmatrix}^{\otimes d}.
\end{equation*}
It follows that
\begin{equation*}
U_Y\big((32m-2)\frac{\pi}{8}\big) = U_Y\big((16m-2)\frac{\pi}{8}\big) = U_Y\big((-2)\frac{\pi}{8}\big) 
= \begin{bmatrix} \frac{1}{\sqrt{2}} &  \frac{\ii}{\sqrt{2}} \\  \frac{\ii}{\sqrt{2}}&  \frac{1}{\sqrt{2}}\end{bmatrix}^{\otimes d}.
\end{equation*}
The direct product $X\times Y$ admits proper $\big( \{a\} \times V(Y)\big)$-fractional revival at time $\pi/8$, for any $a \in V(X)$.
\end{example}

\subsection{Double cover}
Given two graphs, $X$ and $Y$, on the same vertex set $V$, we use $X\ltimes Y$ to denote the graph on vertex set $V\times \{0,1\}$ with adjacency matrix
\begin{equation*}
A_{X\ltimes Y} = \begin{bmatrix} A_X & A_Y\\A_Y& A_X\end{bmatrix}.
\end{equation*}
When $A_X=J-I-A_Y$, we have a double cover of the complete graph and $X\ltimes Y$ is also called the switching graph of $X$.

By Lemma~5.1 of \cite{CoutinhoGodsilPSTProducts},
\begin{equation*}
U_{X\ltimes Y} = \frac{1}{2} \begin{bmatrix} e^{-\ii t (A_X+A_Y)} + e^{-\ii t (A_X-A_Y)} & e^{-\ii t (A_X+A_Y)} - e^{-\ii t (A_X-A_Y)}\\
e^{-\ii t (A_X+A_Y)} - e^{-\ii t (A_X-A_Y)} & e^{-\ii t (A_X+A_Y)} + e^{-\ii t (A_X-A_Y)}\end{bmatrix}
\end{equation*}
which is equal to
\begin{equation*}
U_{X\ltimes Y} = \frac{1}{2} \begin{bmatrix} U_X(t) \left(U_Y(t)+U_Y(-t) \right) & U_X(t) \left(U_Y(t)-U_Y(-t) \right)\\
U_X(t) \left(U_Y(t)-U_Y(-t) \right)&U_X(t) \left(U_Y(t)+U_Y(-t) \right)\end{bmatrix}
\end{equation*}
if $A_X$ and $A_Y$ commute.   If there is a time $\tau$ when $U_Y(\tau)=U_Y(-\tau)$, then $X\ltimes Y$ has $\left(V\times \{0\}\right)$-fractional revival.

\begin{example}
\label{EX:Double}
Let $Y$ be the line graph of $K_{n,n}$ for some odd $n$.  We have 
\begin{equation*}
U_Y(\pi) = U_Y(-\pi) = I-2 E,
\end{equation*}
where $E$ is the orthogonal projection onto the $(n-2)$-eigenspace of $A_Y$.
Let $X$ be the complement of $Y$. Since $Y$ is regular, $A(X)$ and $A(Y)$ commute
and
\begin{equation*}
U_{X\ltimes Y} = \begin{bmatrix} U_X(\frac{\pi}{2}) U_Y(\frac{\pi}{2}) & \0 \\ \0 & U_X(\frac{\pi}{2}) U_Y(\frac{\pi}{2}) \end{bmatrix}.
\end{equation*}
\end{example}

\subsection{Join}
Let $X$ be a connected $k$-regular graph on $n$ vertices and $Y$ be a connected $h$-regular graph on $m$ vertices.
The join of $X$ and $Y$ is the graph, $X+Y$, that has adjacency matrix 
\begin{equation*}
A_{X+Y}= \begin{bmatrix} A_X & J_{n\times m} \\ J_{m\times n} & A_Y\end{bmatrix}.
\end{equation*}
Let
\begin{equation*}
A_X=\sum_{r=0}^{d_1} \lambda_r M_r  \quad \text{and} \quad 
A_Y=\sum_{s=0}^{d_2} \mu_s N_s
\end{equation*}
be the spectral decomposition of $A_X$ and $A_Y$, respectively, with $\lambda_0=k$, $M_0=\frac{1}{n}J_n$, $\mu_0=h$ and $N_0=\frac{1}{m}J_m$.
Then $A$ has the spectral decomposition
\begin{align*}
A_{X+Y} =&  \sum_{r=1} \lambda_r \begin{bmatrix} M_r & \0 \\ \0 & \0\end{bmatrix} + 
\sum_{s=1} \mu_s \begin{bmatrix} \0 & \0 \\ \0 & N_s\end{bmatrix} + \\ & +
\theta_1 \begin{bmatrix}\frac{\alpha_1^2}{\alpha_1^2 n + m}   J_n & \frac{\alpha_1}{\alpha_1^2 n + m}  J_{n\times m} \\\frac{\alpha_1}{\alpha_1^2 n + m}  J_{m\times n} &\frac{1}{\alpha_1^2 n + m}  J_m \end{bmatrix} +
\theta_2 \begin{bmatrix} \frac{\alpha_2^2}{\alpha_2^2 n + m}  J_n &\frac{\alpha_2}{\alpha_2^2 n + m} J_{n\times m} \\ \frac{\alpha_2}{\alpha_2^2 n + m} J_{m\times n}  &\frac{1}{\alpha_2^2 n + m}  J_m \end{bmatrix} ,
\end{align*}
where $\alpha_1$ and $\alpha_2$ are roots of the quadratic polynomial $nx^2- (k-h)x-m=0$,  and
$\theta_j = n\alpha_j+h$, for $j=1,2$.
Then
\begin{eqnarray*}
U_{X+Y}(t) &=& \sum_{r=1} e^{-\ii t \lambda_r} \begin{bmatrix} M_r & \0 \\ \0 & \0\end{bmatrix} + 
\sum_{s=1} e^{-\ii t \mu_s} \begin{bmatrix} \0 & \0 \\ \0 & N_s\end{bmatrix} +\\
&&
+ e^{-\ii t \theta_1} \begin{bmatrix}\frac{\alpha_1^2}{\alpha_1^2 n + m}   J_n & \frac{\alpha_1}{\alpha_1^2 n + m}  J_{n\times m} \\\frac{\alpha_1}{\alpha_1^2 n + m}  J_{m\times n} &\frac{1}{\alpha_1^2 n + m}  J_m \end{bmatrix} +
e^{-\ii t \theta_2} \begin{bmatrix} \frac{\alpha_2^2}{\alpha_2^2 n + m}  J_n &\frac{\alpha_2}{\alpha_2^2 n + m} J_{n\times m} \\ \frac{\alpha_2}{\alpha_2^2 n + m} J_{m\times n} &\frac{1}{\alpha_2^2 n + m}  J_m \end{bmatrix}.
\end{eqnarray*}
When $\tau = \frac{2\pi}{\theta_2-\theta_1}$,
\begin{equation*}
e^{-\ii \tau \theta_1}  \frac{\alpha_1}{\alpha_1^2 n + m} +e^{-\ii \tau \theta_2} \frac{\alpha_2}{\alpha_2^2 n + m}  = 0,
\end{equation*}
so $X+Y$ has $V(X)$-fractional revival.

\subsection{Antipodal distance-regular $r$-fold cover of $K_n$}
\label{Section:DRACKn}

A $r$-fold cover of $K_n$ is a graph $X$ in which every vertex of $K_n$ is replaced by a set of $r$ vertices, called a fibre,
with a perfect matching between any two fibres.  
In this example, we consider antipodal distance regular $r$-fold cover of $K_n$ whose spectrum is determined by
the parameters $n$, $r$ and the number, $c$, of common neighbours between two vertices at distance two.
The eigenvalues of $X$ are
\begin{equation*}
\theta_1=n-1, \ \theta_2=-1, \ \theta_3=\frac{\delta+\sqrt{\delta^2+4(n-1)}}{2}, \ \text{and}\ \theta_4=\frac{\delta-\sqrt{\delta^2+4(n-1)}}{2},
\end{equation*}
where $\delta=n-2-rc$.
The direct sum of the eigenspaces corresponding to $\theta_1$ and $\theta_2$ consists of all vectors that are constant on the fibres.
If  the rows and columns of the adjacency matrix $A$ of $X$ are arranged so that it has $n\times n$ blocks of size $r\times r$ corresponding to the fibres, then
the principal idempotents of $A$  satisfy
\begin{equation*}
E_{1}+E_{2} = I_n \otimes \frac{1}{r} J_r \quad \text{and}\quad E_{3}+E_{4} =I_{rn}-(E_1+E_2)= I_n\otimes \big(I_r-\frac{1}{r}J_r\big).
\end{equation*}
When $\delta=2$, $n=4m^2$ and $\tau=\frac{\pi}{2m}$, we have
\begin{equation*}
e^{-\ii \theta_1\tau}=e^{-\ii \theta_2\tau} = e^{\frac{\pi \ii}{2m}}
\quad \text{and}\quad
e^{-\ii\theta_3\tau}=e^{-\ii\theta_4\tau}=e^{\frac{(2m-1)\pi \ii}{2m}}.
\end{equation*}
Similarly, when $\delta=-2$, $n=4m^2$ and $\tau=\frac{\pi}{2m}$, we have
\begin{equation*}
e^{-\ii \theta_1\tau}=e^{-\ii \theta_2\tau} = e^{\frac{\pi \ii}{2m}}
\quad \text{and}\quad
e^{-\ii\theta_3\tau}=e^{-\ii\theta_4\tau}=e^{\frac{(2m+1)\pi \ii}{2m}}.
\end{equation*}
For both cases,
\begin{equation*}
U(\tau) = I_n \otimes  \big(  e^{-\ii\theta_3\tau} I_r + \frac{e^{-\ii \theta_1\tau}-e^{-\ii\theta_3\tau}}{r} J_r \big)
\end{equation*}
and $K$-fractional revival occurs in $X$, for each fibre $K$.

Please see \cite{KlinPech} for the background and some constructions of distance regular antipodal covers of $K_n$ for which $\delta=\pm 2$.
Section~6.1 of \cite{KlinPech} gives an example that has parameters $n=36$, $r=3$ and $c=12$, admitting $K$-fractional revival at time $\frac{\pi}{6}$.

\section{Commuting partitions and ratio condition}
\label{Section:P-ratio}

In Section~\ref{Section:DRACKn}, we get $K$-fractional revival at time $\tau$ by partitioning the eigenspaces of $A$ into two classes such that
\begin{itemize}
\item
the sum of the $E_r$'s in each class has the desired block diagonal structure, and
\item
$e^{-\ii \tau \theta_r} = e^{-\ii \tau \theta_s}$, for $r$ and $s$ in the same class of the partition.
\end{itemize}
We see that the above approach works in general.

\begin{theorem}
\label{Thm:Char1}
Let $A$ be the adjacency matrix of a weighted graph $X$ with spectral decomposition
\begin{equation*}
A = \sum_{r=0}^d \theta_r E_r.
\end{equation*}
For $K\subset V(X)$, $K$-fractional revival occurs in $X$ at time $\tau$ if and only if there exists a partition $P$ of $\{0,1,\ldots,d\}$ 
such that
\begin{enumerate}[(i)]
\item
\label{Cond:Char1i}
$\big(\sum_{r\in C} E_r \big) D_K=D_K \big(\sum_{r\in C} E_r \big)$, for each class $C$ of $P$,  and
\item
\label{Cond:Char1ii}
$e^{-\ii \tau \theta_r} = e^{-\ii \tau \theta_s}$ whenever  $r$ and $s$ belong to the same class of $P$.
\end{enumerate}
\end{theorem}
\begin{proof}
It is straightforward to see that if there exists a partition $P$ and a time $\tau$ satisfying the above conditions,
then $U(\tau) D_K=D_K U(\tau)$ and $K$-fractional revival occurs.

Conversely, suppose $K$-fractional revival occurs at time $\tau$.
Let $P$ be the partition of $\{0,1,\ldots,d\}$ where $r$ and $s$ belong to the same class if and only if 
$e^{-\ii\tau\theta_r}=e^{-\ii\tau\theta_s}$.  
For each class $C$ of $P$, let $\lambda_C = e^{-\ii\tau\theta_r}$ for some $r\in C$.  Then
\begin{equation*}
U(\tau)= \sum_{C\in P} \lambda_C \big(\sum_{r\in C} E_r \big).
\end{equation*}
By Lagrange interpolation, we can express the matrix $\sum_{r\in C} E_r$ as a polynomial in $U(\tau)$, for each class $C$.
Since $U(\tau)$ commutes with $D_K$, so does the matrix $\sum_{r\in C} E_r$, for  each class $C$, so (\ref{Cond:Char1i}) holds.
\end{proof}

Note that if $P$ has one class $C=\{0,1,\ldots,d\}$ then
$\sum_{r\in C} E_r =I_n$
and Condition~(\ref{Cond:Char1i}) holds for any subset $K$ of $V(X)$.   
Hence  the partition in Theorem~\ref{Thm:Char1} has at least two classes for proper  $K$-fraction revival to occur.

\begin{definition}
\label{Def:CommutePartition}
Given the spectral decomposition $A=\sum_{r=0}^d \theta_rE_r$, we say a partition $P$ of $\{0,1,\ldots,d\}$ \textsl{commutes with $D_K$}
if 
\begin{equation*}
\big(\sum_{r\in C} E_r \big) D_K=D_K \big(\sum_{r\in C} E_r \big),
\end{equation*}
for each class $C$ of $P$.
\end{definition}
Given two partitions $P$ and $Q$ of $\{0,1,\ldots,d\}$, we say $P \leq Q$ if $P$ is a refinement of $Q$, that is,
every class of $P$ is a subset of some class of $Q$.
The \textsl{meet of $P$ and $Q$}, $P\land Q$, is the partition with classes being the non-empty intersections of the classes of $P$ and $Q$.
If $C_1 \cap C_2$ is a class of $P\land Q$ then
\begin{equation*}
\sum_{r\in C_1 \cap C_2} E_r = \big(\sum_{r\in C_1} E_r \big)  \big(\sum_{s\in C_2} E_s\big).
\end{equation*}
Hence if both $P$ and $Q$ commute with $D_K$ then so does $P \land Q$.    
Consequently, there exists a unique minimal partition that commutes with $D_K$, we use $P_{min}^K$ to denote this partition.

\begin{theorem}\label{thm:inequalityandparallel}
    Let $A$ be the adjacency matrix of a weighted graph $X$ with spectral decomposition $A = \sum_{r=0}^d \theta_r E_r$. Let $K \subseteq V(X)$ and consider the minimal partition $P_{min}^K = \{C_1,...,C_\ell\}$ that commutes with $D_K$. Let $E_{C_j} = \sum_{r \in C_j} E_r$. Then
    \[|\{j : E_{C_j}D_K \neq 0\}| \leq |K|.\]
    If equality holds, then, for all $r$, the columns of $E_r$ corresponding to vertices in $K$ are parallel vectors.
\end{theorem}
\begin{proof}
    For some $C \in P_{min}^K$, note that $D_K E_C$ maps vectors to a subspace of dimension $|K|$, determined by the entries corresponding to the vertices in $K$. This is a orthogonal projection, as $D_K$ and $E_C$ commute, and if $C_i \neq C_j$, $D_K E_{C_i} D_K E_{C_j} = 0$. Thus the images of $D_K E_{C_j}$ are all orthogonal subspaces, whence the inequality follows.
    
    When equality holds, it follows that if $D_K E_C \neq 0$, then $D_K E_C$ has rank one. Because 
    \[D_K E_C = \sum_{r \in C} D_K E_r ,\]
    and the $E_r$ are projections onto orthogonal subspaces, it follows that each $D_K E_r$ has rank one or zero, as we wanted to show.
\end{proof}

Let us now characterize when fractional revival happens with respect to an eigenvalue condition associated to the minimal commuting partition.

\begin{definition}
\label{Def:Ratio}
The eigenvalues $\theta_0, \theta_1,\ldots, \theta_d$ of $A$ satisfy the \textsl{ratio condition} with respect to a partition $P$ of $\{0,1,\ldots,d\}$ if,
for any  $r$ and $s$ in the same class of $P$ and distinct elements $h$ and $k$ in the same class of $P$, the ratio
\begin{equation*}
\frac{\theta_r -\theta_s}{\theta_h-\theta_k}
\end{equation*}
is rational.
\end{definition}

Here is a characterization of $K$-fractional revival in terms of $P_{min}^K$ and the ratio condition.
Note that the definition of $K$-fractional revival includes periodicity at all vertices in $K$.
\begin{theorem}
\label{Thm:Char2}
Let $A$ be the adjacency matrix of a weighted graph $X$ with spectral decomposition
\begin{equation*}
A = \sum_{r=0}^d \theta_r E_r.
\end{equation*}
For $K \subset V(X)$, $K$-fractional revival occurs in $X$ if and only if the eigenvalues of $A$ satisfy the ratio condition with respect to $P_{min}^K$.
\end{theorem}
\begin{proof}
There exists a time $\tau$ such that $e^{-\ii\tau\theta_r}=e^{-\ii\tau\theta_s}$, for all $r$ and $s$ in the same class of $P_{min}^K$,
if and only if $\theta_0, \ldots, \theta_d$ satisfy the ratio condition.
\end{proof}

\section{Minimum commuting partition}
\label{Section:Pmin}

In this section, we investigate the minimum partition, $P_{min}^K$, commuting with $D_K$  as it plays an important role in the characterization of $K$-fractional revival in $X$.
For matrix $M$ with rows and columns indexed by the vertices in $X$,  we use
$\wt{M}$ to denote the  submatrix of $M$ with rows and columns indexed by the elements of $K$.
For a vector $v \in \C^{V(X)}$, we use $\tp{v}$ and $\bttm{v}$ to denote the restriction of  $v$ to the elements in $K$ and the restriction of $v$ to the elements in $V(X)\backslash K$, respectively.

Suppose $P_{min}^K = \{C_1, \ldots, C_z\}$.   For $j=1,\ldots, z$,  let
\begin{equation*}
F_j= \sum_{r\in C_j} E_r.
\end{equation*}
It follows from the definition of $P_{min}^K$ that
\begin{equation*}
F_j = \begin{bmatrix} \wt{F_j} & \0 \\ \0 & F_j'\end{bmatrix},
\end{equation*}
for some $(n-|K|)\times (n-|K|)$ idempotent matrix $F_j'$.
The $F_j$'s satisfy $\sum_{j=1}^z F_j = I_{n}$,
\begin{equation*}
F_j F_h = \delta_{j,h} F_j, \quad \forall h, j,
\end{equation*}
and
\begin{equation}
\label{Eqn:FE}
F_j E_r = E_r F_j=
\begin{cases} 
E_r & \text{if $r\in C_j$,}\\
\0 & \text{otherwise.}
\end{cases}
\end{equation}
Restricting to the set $K$ yields 
\begin{eqnarray}
\label{Eqn:K-SumF}
&\sum_{j=1}^z \wt{F_j} = I_{|K|},\\
\label{Eqn:K-Idem}
&\wt{F_j}  \wt{F_h} = \delta_{j,h} \wt{F_j}, \quad \forall h, j,
\end{eqnarray}
and
\begin{equation}
\label{Eqn:K-FE}
\wt{F_j} \wt{E_r} = \wt{E_r} \wt{F_j}=
\begin{cases} 
\wt{E_r} & \text{if $r\in C_j$,}\\
\0 & \text{otherwise.}
\end{cases}
\end{equation}

\begin{proposition}
\label{Prop:ZeroBlock}
Let $r\in C_j$.  Then 
 $\wt{F_j}\neq \0$ if and only if $\wt{E_r}\neq \0$.
\end{proposition}
\begin{proof}
Each principal idempotent $E_r$ is positive semidefinite, so there exists a matrix $B_r$ such that $E_r= B_r^T B_r$. 
So $(E_r)_{w,w} \geq 0$, for $w\in V(X)$.
If $(F_j)_{w,w}=0$ then $(E_r)_{w,w}=0$ and both $e_w^TE_r$ and $E_re_w$ are zero vectors.  
Hence, $\wt{F_j}= \0$ implies $\wt{E_r}=\0$.

If $\wt{E_r}=\0$ then the first $|K|$ columns of $B_r$ are zero, and
the first $|K|$ rows and columns  of $E_r$ are zero.   In this case $E_r D_K=D_KE_r$ and $\{r\}$ is a class of $P_{min}^K$.   As a result,
$\wt{F_j} = \wt{E_r} = \0$.
\end{proof}

\begin{proposition}
\label{Prop:Connected}
Let $X$ be a connected graph with at least two non-zero $\wt{F_j}$'s.  Then all of the non-zero $\wt{F_j}$'s are non-diagonal.
\end{proposition}
\begin{proof}
If $\wt{F_j}$ is a diagonal matrix then the Equation~(\ref{Eqn:K-Idem}) implies that $\wt{F_j}$ has only $0$ or $1$ in its diagonal entries. From Equation~\eqref{Eqn:K-SumF} and the fact that $\wt{F_\ell}$s are positive semidefinite, any other $\wt{F_\ell}$ must have $0$ diagonal entries, and thus an entire block of $0$s, corresponding to the diagonal entries of $\wt{F_j}$ equal to $1$.

For any $r\in C_\ell$, $(E_r)_{a,b}=0$ when $(\wt{F_\ell})_{a,a}=0$ or $(\wt{F_\ell})_{b,b}=0$. Thus all $\wt{E_r}$s are block diagonal matrices, up to permutation of rows and columns, with their blocks restricted either to rows and columns where $\wt{F_j}$ has non-zero diagonal entries if $j = \ell$, or to rows and columns where $\wt{F_j}$ has zero diagonal entries if $j \neq \ell$.

Hence $\wt{A^k}$ is a block diagonal matrix, for all $k\geq 0$, which implies that $X$ is not a connected graph.
\end{proof}

By Equation~(\ref{Eqn:K-SumF}),
if a graph has only one non-zero $\wt{F_j}$
then $\wt{F_j}=I_{|K|}$.
The following example is a connected graph with exactly one non-zero $\wt{F_j}$.
\begin{example}
\label{Ex:DiagF}
The graph $X$
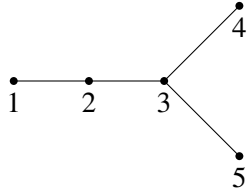
\begin{center}
\begin{tikzpicture}
\fill (0,0) circle (1.5pt);
\draw (0,0) node[anchor=north]{\small $1$};
\fill (1,0) circle (1.5pt);
\draw (1,0) node[anchor=north]{\small $2$};
\fill (2,0) circle (1.5pt);
\draw (2,0) node[anchor=north]{\small $3$};
\fill (3,1) circle (1.5pt);
\draw (3,1) node[anchor=north]{\small $4$};
\fill (3,-1) circle (1.5pt);
\draw (3,-1) node[anchor=north]{\small $5$};

\draw (0,0)--(1,0)--(2,0);
\draw (2,0)--(3,1);
\draw (2,0)--(3,-1);
\end{tikzpicture}
\end{center}
has spectrum
\[
\theta_1=\sqrt{2+\sqrt{2}}, \quad
\theta_2=-\sqrt{2+\sqrt{2}}, \quad
\theta_3=\sqrt{2-\sqrt{2}}, \]
\[
\theta_4=-\sqrt{2-\sqrt{2}}, \quad\text{and}\quad
\theta_5=0.
\]
Let $K=\{1,2,3\}$.  The partition $P_{min}^K$ has two classes $C_1=\{1,2,3,4\}$ and $C_2=\{5\}$, with
\begin{equation*}
F_1=\begin{bmatrix} 1& 0 & 0&0&0\\0&1&0&0&0\\0&0&1&0&0\\0&0&0&0.5&0.5\\0&0&0&0.5&0.5\end{bmatrix}
\quad \text{and}\quad
F_2=\begin{bmatrix} 0& 0 & 0&0&0\\0&0&0&0&0\\0&0&0&0&0\\0&0&0&0.5&-0.5\\0&0&0&-0.5&0.5\end{bmatrix},
\end{equation*}
so
\begin{equation*}
\wt{F_1} = I_3 \quad \text{and}\quad 
\wt{F_2}=\0.
\end{equation*}
\end{example}

Theorem~\ref{Thm:Char2} states that the ratio condition with respect to $P_{min}^K$ must be satisfied if $K$-fractional revival occurs in $X$.
We also see in Section~\ref{Section:K-FR}, from the viewpoint of quantum walk on density matrices, that the condition
\begin{equation}
\label{Eqn:Ratio}
\frac{\theta_r-\theta_s}{\theta_h-\theta_k} \in \Q, \quad \forall (\theta_r,\theta_s), (\theta_h,\theta_k) \in \Phi_K  \text{where $\theta_h\neq \theta_k$}
\end{equation}
is necessary for $K$-fractional revival.   We now show that these two ratio conditions are equivalent.

\begin{lemma}
\label{Lem:K-support}
If $(\theta_r,\theta_s) \in \Phi_K$ then $r$ and $s$ belong to the same class, say $C_j$, of $P_{min}^K$, where $\wt{F_j} \neq \0$.
\end{lemma}
\begin{proof}
Suppose $r \in C_h$ and $s\in C_j$.
Equations~(\ref{Eqn:FE}) and (\ref{Eqn:K-Idem}) give
\begin{equation*}
E_r D_K E_s = E_r F_h D_K F_j E_s 
= E_r \begin{bmatrix} \wt{F_h}\wt{F_j} & \0 \\ \0 & \0 \end{bmatrix} E_s= \delta_{h,j} E_r  \begin{bmatrix}\wt{F_j} & \0 \\ \0 & \0 \end{bmatrix} E_s.
\end{equation*}
\end{proof}

\begin{lemma}
\label{Lem:MinClass}
For each class, $C_j$, of $P_{min}^K$, there does not exist a partition $C_j=S_1 \cup S_2$ such that
\begin{equation*}
E_r D_K E_s = \0, \quad \text{for all $r \in S_1$ and $s\in S_2$.}
\end{equation*}
\begin{proof}
Suppose there exists a class, $C_j$, of $P_{min}^K$ that has a partition $C_j=S_1 \cup S_2$ satisfying
\begin{equation*}
E_r D_K E_s = \0, \quad \text{for all $r \in S_1$ and $s\in S_2$.}
\end{equation*}
Let $W_1= \sum_{r\in S_1} E_r$ and $W_2=\sum_{s\in S_2}E_s$.  Then 
\begin{equation*}
W_1 D_K W_2=\0.
\end{equation*}
It follows from $F_j = W_1+W_2$ that
\begin{equation*}
(W_1+W_2) D_K = D_K(W_1+W_2).
\end{equation*}
Multiplying $W_1$ on the left gives
\begin{equation*}
W_1 D_K = W_1 D_K W_1.
\end{equation*}
If we write $W_1$ as
\begin{equation*}
\begin{bmatrix}
\wt{W_1} & B\\ B^T & C
\end{bmatrix},
\end{equation*}
then the above equation implies $B^T B = \0$ and $W_1 D_K=D_KW_1$.
As $P_{min}^K$ is the minimum partition commuting with $D_K$, we conclude that $W_1=\0$ and $S_1=\emptyset$.
\end{proof}
\end{lemma}

\begin{theorem}
\label{thm:Ratio}
Let $\{\theta_0,...,\theta_d\}$ be the distinct eigenvalues of the graph. Given $K \subseteq V(X)$, recall the definitions of $P_{min}^K$ and $\Phi_K$. The following are equivalent.
\begin{enumerate}[(a)]
    \item 
    For all $(\theta_r,\theta_s), (\theta_h,\theta_k) \in \Phi_K$, with $\theta_h \neq \theta_k$, it follows that \label{thmratio:conditiona}
    \[\frac{\theta_r-\theta_s}{\theta_h-\theta_k} \in \Q.\]
    \item 
    \label{thmratio:conditionb}
    For any $C_{j_1},C_{j_2} \in P_{min}^K$, and for all $\theta_r,\theta_s \in C_{j_1}$ and all $\theta_h,\theta_k \in C_{j_2}$, with $\theta_h \neq \theta_k$, it follows that
    \[\frac{\theta_r-\theta_s}{\theta_h-\theta_k} \in \Q.\]
\end{enumerate}
\end{theorem}
\begin{proof}
By Lemma~\ref{Lem:MinClass},
for distinct $r$ and $s$ in $C_{j_1}$,
there exist distinct elements $r_1, \ldots, r_{m_1}$ in $C_{j_1}$ such that $r_1=r$, $r_{m_1}=s$ and
\begin{equation*}
(\theta_{r_i}, \theta_{r_{i+1}})  \in \Phi_K, \quad \text{for $i=1\ldots,m_1-1$.}
\end{equation*}
Similarly, for distinct $h$ and $k$ in $C_{j_2}$,
there exist distinct elements $h_1, \ldots, h_{m_2}$ in $C_{j_2}$ such that $h_1=h$, $h_{m_2}=k$ and
\begin{equation*}
(\theta_{h_\ell}, \theta_{h_{\ell+1}})  \in \Phi_K, \quad \text{for $\ell=1\ldots,m_2-1$.}
\end{equation*}

Suppose condition~(\ref{thmratio:conditiona}) holds. Then, for $1\leq i \leq m_1-1$ and $1\leq \ell \leq m_2-1$,
\begin{equation*}
\frac{\theta_{h_\ell}-\theta_{h_{\ell+1}}}{\theta_{r_i}- \theta_{r_{i+1}}}\in \Q.
\end{equation*}
It follows that, for  $1\leq i \leq m_1-1$, both
\begin{equation*}
\frac{\theta_h-\theta_k}{\theta_{r_i}- \theta_{r_{i+1}}} = \frac{\sum_{\ell=1}^{m_2-1}(\theta_{h_\ell}- \theta_{h_{\ell+1}})}{\theta_{r_i}- \theta_{r_{i+1}}} 
\quad \text{and}\quad
\frac{\theta_{r_i}-\theta_{r_{i+1}}} {\theta_h-\theta_k}
\end{equation*}
are rational.
Consequently,
\begin{equation*}
\frac{\theta_r-\theta_s}{\theta_h- \theta_k} = \frac{\sum_{i=1}^{m_1-1}(\theta_{r_i}- \theta_{r_{i+1}})}{\theta_h- \theta_k} \in \Q,
\end{equation*}
and condition~(\ref{thmratio:conditionb}) holds.

Conversely, by Lemma~\ref{Lem:K-support} and Proposition~\ref{Prop:ZeroBlock}, condition~(\ref{thmratio:conditionb}) implies condition (\ref{thmratio:conditiona}).
\end{proof}
 
\begin{corollary}
\label{Cor:K-FRratio}
Suppose $X$ has integer weights.
If $K$-fractional revival  occurs in $X$ at time $\tau>0$ then there exists a square-free integer $\Delta$
such that $\theta_r-\theta_s$ is an integer multiple of $\sqrt{\Delta}$, for all $r$ and $s$ in the same class of $P_{min}^K$.
\end{corollary}
\begin{proof}
It follows from Theorem~\ref{Thm:K-FReigenvalues} and Theorem~\ref{thm:Ratio}.
\end{proof}

\section{Block decomposition}
\label{Section:Block}

Suppose $K$-fractional revival occurs in $X$ at time $\tau$, then, up to permuting rows and columns,
\begin{equation*}
U(\tau)=
\begin{bmatrix}
\wt{U(\tau)} &\0\\ \0 & U'
\end{bmatrix},
\end{equation*}
for some $(n-|K|)\times (n-|K|)$ unitary matrix $U'$.

Let $\alg{A}$ denote the ring of polynomials in $A$ over $\C$.
Having a matrix in $\alg{A}$ with the desired block diagonal structure is a necessary condition of $K$-fractional revival in $X$.

\begin{definition}
\label{Def:Decomp}
We say $X$ is  \textsl{decomposable with respect to $K$} if $\alg{A}$ contains a non-identity block diagonal matrix
\begin{equation*}
\begin{bmatrix} H & \0 \\ \0 & H' \end{bmatrix}
\end{equation*} 
where the rows and columns of $H$ are indexed by the elements of $K$.
\end{definition}
When proper $K$-fractional revival occurs at time $\tau$, then  $\wt{U(\tau)}$ is not a diagonal matrix.
We say $X$ is \textsl{properly} decomposable with respect to $K$ if $\alg{A}$ contains a matrix in the above block structure where $H$ is not diagonal.

\begin{theorem}
\label{Thm:Decomp}
A connected graph $X$ is properly decomposable with respect to $K$ if and only if $P_{min}^K$ has more than one class and at least two of $\wt{F_j}$'s are non-zero.
\end{theorem}
\begin{proof}
Suppose $X$ is connected with at least two non-zero $\wt{F_j}$'s. 
It follows from Proposition~\ref{Prop:Connected} that, without loss of generality, $\wt{F_1}$ is not a diagonal matrix.  As $F_1$ is a matrix in $\alg{A}$ that commutes with $D_K$, we conclude that 
$X$ is properly decomposable with respect to $K$.

Conversely, by Equation~(\ref{Eqn:K-SumF}), we can assume without loss of generality that $\wt{F_1}=I_{|K|}$ and all other $\wt{F_j}$'s are the zero matrix.
Then each block diagonal matrix  
\begin{equation*}
\begin{bmatrix} H & \0 \\ \0 & H' \end{bmatrix} 
\end{equation*}
in $\alg{A}$ has $H \in \spn\{I_{|K|}\}$, so $X$ is not properly diagonalizable.
\end{proof}

\begin{proposition}
\label{Prop:Decomp}
Let $H$ be a $|K| \times |K|$ matrix and $H'$ be an $(n-|K|)\times (n-|K|)$ matrix. Then the following are equivalent.
\begin{enumerate}[i.]
\item
\label{Cond:Decompi}
The block diagonal matrix
\begin{equation*}
\begin{bmatrix} H & \0 \\ \0 & H' \end{bmatrix} \in \alg{A}.
\end{equation*}
\item
\label{Cond:Decompii}
For any eigenvector $v$ of $A$, there exists a real number $\lambda$ such that  $H \tp{v} = \lambda \tp{v}$ and
$H' \bttm{v} = \lambda \bttm{v}$.  (At most one of $\tp{v}$ and $\bttm{v}$ can be the zero vector.)
\end{enumerate}
\end{proposition}
\begin{proof}
Since the eigenvectors of $A$ are eigenvectors of every matrix in $\alg{A}$, 
(\ref{Cond:Decompi}) implies (\ref{Cond:Decompii}).

Conversely,  we assume (\ref{Cond:Decompii}) holds.  
Let $u$ and $v$ eigenvectors of $A$ in the same eigenspace.   If both $\tp{u}$ and $\tp{v}$ are non-zero, then any non-zero linear combination of $\tp{u}$ and $\tp{v}$
is an eigenvector of $H$.  We conclude that $\tp{u}$ and $\tp{v}$ belong to the same eigenspace of $H$.
Similarly, if both $\bttm{u}$ and $\bttm{v}$ are non-zero, then they belong to the same eigenspace of $H'$.

Given the spectral decomposition $A=\sum_{r=0}^d \theta_r E_r$,  we define $\lambda_r$, for $r=0,\ldots, d$, to be the real number satisfying
\begin{equation*}
H \tp{(E_r e_a)} = \lambda_r \tp{(E_r e_a)}
\quad \text{and} \quad 
H' \bttm{(E_r e_a)} = \lambda_r \bttm{(E_r e_a)}, 
\quad \forall a\in V(X).
\end{equation*}
Suppose the list $\lambda_0, \lambda_1, \ldots, \lambda_d$ has $z$ distinct values $\mu_1, \ldots, \mu_z$.
For $j=1,\ldots, z$,  let $M_j$ be the $|K|\times |K|$ orthogonal projection matrix onto the space
\begin{equation*}
\spn \left\{ \tp{(E_r e_a)} : a\in V(X) \ \text{and}\ \lambda_r=\mu_j \right\},
\end{equation*}
and  let $N_j$ be the $(n-|K|) \times (n-|K|)$ orthogonal projection matrix onto the space
\begin{equation*}
\spn \left\{ \bttm{(E_r e_a)} : a\in V(X) \ \text{and}\ \lambda_r=\mu_j \right\}.
\end{equation*}
(Note that some of these $M_j$'s and $N_j$'s could be the zero matrix.)
We have
\begin{equation*}
H M_j = \mu_j M_j
\quad \text{and} \quad
H' N_j = \mu_j N_j.
\end{equation*}
For $j=1,\ldots,z$, we have
\begin{equation*}
\begin{bmatrix} M_j & \0 \\ \0 & N_j \end{bmatrix} E_r
=
\begin{cases}
E_r & \text{if $\lambda_r=\mu_j$,}\\
\0 & \text{otherwise.}
\end{cases}
\end{equation*}
For $r=0,\ldots, d$, if $\lambda_r=\mu_j$ then
\begin{eqnarray*}
\begin{bmatrix} H & \0 \\ \0 & H' \end{bmatrix} E_r
&=&\begin{bmatrix} H & \0 \\ \0 & H' \end{bmatrix}  \begin{bmatrix} M_j & \0 \\ \0 & N_j \end{bmatrix} E_r\\
&=& \lambda_r \begin{bmatrix} M_j & \0 \\ \0 & N_j \end{bmatrix} E_r\\
&=& \lambda_r E_r.
\end{eqnarray*}
Consequently,
\begin{equation*}
\begin{bmatrix} H & \0 \\ \0 & H' \end{bmatrix}  = \sum_{r=0}^d \lambda_r E_r \in \alg{A}.
\end{equation*}
\end{proof}

If $K$-fractional revival occurs in $X$ at time $\tau$, then $U(\tau)$ is a block diagonal matrix in $\alg{A}$
that commutes with all non-negative powers of $A$.   Thus, $\wt{U(\tau)}$ is a symmetric
unitary matrix commuting with $\wt{A^k}$, for $k\geq 0$. 
We give a characterization of symmetric unitary matrices that commute with $\wt{A^k}$, for $k\geq 0$,
which is used in Section~\ref{Section:FC}.

\begin{lemma}
\label{Lem:Unitary}
A symmetric unitary matrix $H$ has an orthonormal basis of real eigenvectors.
\end{lemma}
\begin{proof}
Since $H$ is symmetric, there exist real symmetric matrices $A$ and $B$ such that
\begin{equation*}
H=A+ \ii B.
\end{equation*}
Now $H^* = A-\ii B$ and
\begin{equation*}
I=H^* H = (A^2+B^2) + \ii (AB - BA)
\end{equation*}
implies $AB=BA$.  Hence $A$ and $B$ are diagonalizable by the same set of real orthonormal eigenvectors,
which are also eigenvectors of $H$.
\end{proof}

\begin{proposition} 
\label{Prop:Commute}
Let $H$ be a $|K|\times |K|$ symmetric unitary matrix. The following statements are equivalent.
\begin{enumerate}[i.]
\item
\label{Cond:Commute1}
$H$ commutes with $\wt{A^k}$, for all $k\geq 0$.
\item
\label{Cond:Commute2}
$H$ commutes with $\wt{E_r}$, for each principal idempotent $E_r$ of $A$.
\item
\label{Cond:Commute3}
$A$ has  an orthogonormal basis of real eigenvectors $v_1, \ldots, v_n$ such that, for $j=1,\ldots,n$, either $\tp{v_j}=\0$ or it is an eigenvector of $H$.
\end{enumerate}
\end{proposition}
\begin{proof}
The statements (\ref{Cond:Commute1}) and (\ref{Cond:Commute2}) are equivalent because $A^k$ is a linear combination of the $E_r$'s,
and each $E_r$ is a polynomial in $A$.

Statement~(\ref{Cond:Commute3}) implies $H$ commutes with $\tp{v_j}\tp{v_j}^T$, for $j=1,\ldots,n$.
For each $r$,  we can write
\begin{equation*}
E_r = \sum_{j : Av_j=\theta_r v_j} v_j v_j^T
\end{equation*}
and 
\begin{equation*}
\wt{E_r} = \sum_{j : Av_j=\theta_r v_j} \tp{v_j} \tp{v_j}^T.
\end{equation*}
As a result, $H$ commutes with $\wt{E_r}$, for all $r$ and (\ref{Cond:Commute2}) holds.

Assume Statement~(\ref{Cond:Commute2}) holds.  Let $H =\sum_{h=1}^z \lambda_h M_h$ be the spectral decomposition of $H$.
By Lemma~\ref{Lem:Unitary}, each $M_h$ is real and symmetric.   Since $M_h$ is a polynomial in $H$, it commutes with $\wt{E_r}$ for all $r$.

For $h=1,\ldots, z$, let $V_h$ be the subspace of $\C^{V(X)}$ spanned by
\begin{equation*}
\left\{ E_r \begin{bmatrix} M_h e_a\\ \0 \end{bmatrix} :  a \in K, r=0,1\ldots, d \right\}.
\end{equation*}
For $h\neq l$, 
\begin{equation*}
\langle E_r \begin{bmatrix} M_l e_a\\ \0 \end{bmatrix}, E_r \begin{bmatrix} M_h e_b\\ \0 \end{bmatrix} \rangle = e_a^T M_l \wt{E_r} M_h e_b
\end{equation*} 
which is equal to zero since $M_h$ commutes with $\wt{E_r}$ and $M_lM_h=\0$.
We see that $V_h$ is orthogonal to $V_l$ whenever $h\neq l$.

We choose a real orthonormal basis of $V_h$, for each $h$, and extend from the union of these bases to a real orthonormal basis $v_1, \ldots, v_n$ of $\C^{V(X)}$.

If $v_j \in V_h$ and $\tp{v_j} \neq \0$ then 
\begin{equation*}
\tp{v_j} \in \spn \left\{ \wt{E_r}M_h e_a : a \in K, r=0,1\ldots, d \right\}.
\end{equation*}
As $\wt{E_r} M_h = M_h\wt{E_r}$, for all $r$, $\tp{v_j}$ lies in the column space of $M_h$ and is  an eigenvector of $H$.

Suppose $v_j$ does not belong to $V_h$, for any $h$.  Then
\begin{equation*}
v_j^T E_r \begin{bmatrix} M_h\\ \0 \end{bmatrix}  = \0
\end{equation*}
for all $r$ and $h$.   In particular, 
\begin{equation*}
\0=v_j^T \left(\sum_{r=0}^d \sum_{h=1}^z E_r \begin{bmatrix} M_h\\ \0 \end{bmatrix}\right) = v_j^T I_n \begin{bmatrix} I_{|K|} \\ \0\end{bmatrix} =  \begin{bmatrix} \tp{v_j} \\ \0\end{bmatrix} .
\end{equation*}
In this case, $\tp{v_j}=\0$.
We conclude that (\ref{Cond:Commute2}) implies (\ref{Cond:Commute3}).
\end{proof}

\section{Fractionally cospectral vertices}
\label{Section:FC}

For the rest of this note, we focus on the decomposability of $X$ with respect to $K=\{a,b\}$ and revisit fractional revival between two vertices.
We introduce the notions of fractional cospectrality and strongly fractional cospectrality of two vertices,
which can be viewed as a generalization of cospectrality and strongly cospectrality, respectively.

Suppose $X$ is a connected graph with real weights that is properly decomposable with respect to $K=\{a,b\}$. Let
$P_{min}^K=\{C_1, \ldots, C_z\}$, and
\begin{equation*}
F_j=\sum_{r\in C_j} E_r, \quad \text{for $j=1,\ldots,z$.} 
\end{equation*}
By Theorem~\ref{thm:inequalityandparallel}, Theorem~\ref{Thm:Decomp} and Proposition~\ref{Prop:Connected}, $z\geq 2$, $\wt{F_j} = 0$ for all but two indices, and exactly two of the $\wt{F_j}$'s, say $\wt{F_1}$ and $\wt{F_2}$, are $2\times 2$ rank one non-diagonal matrices which are projections onto orthogonal subspaces.

Hence we assume, without loss of generality, that
\begin{equation*}
\wt{F_1} = \begin{bmatrix} p \\ q \end{bmatrix}\begin{bmatrix} p & q \end{bmatrix}
\quad \text{and} \quad
\wt{F_2} = \begin{bmatrix} -q \\ p \end{bmatrix}\begin{bmatrix} -q & p \end{bmatrix},
\end{equation*}
for some non-zero real numbers $p$ and $q$ satisfying $p^2+q^2=1$.

\begin{theorem}
\label{Thm:FC}
Let $X$ be a graph with real weights.
If $X$ is properly decomposable with respect to $K=\{a,b\}$ then the $\wt{E_r}$'s commute with each other.
\end{theorem}
\begin{proof}
We see in Proposition~\ref{Prop:ZeroBlock} that 
$\wt{E_r} =\0$ if  and only if $r \not \in C_1 \cup C_2$.

For $r \in C_1$, the column space of $\wt{E_r}$ is a subspace of $\spn\left\{ \begin{bmatrix} p \\ q \end{bmatrix}\right\}$.
Then there exist $\alpha, \alpha' \in \R$ such that
\begin{equation*}
\wt{E_r} = \begin{bmatrix} \alpha p & \alpha' p\\ \alpha q &\alpha'q\end{bmatrix}.
\end{equation*}
Since $\wt{E_r}$ is symmetric, we have $\alpha'=\frac{\alpha q}{p}$ and $\wt{E_r} = \frac{\alpha}{p} \wt{F_1}$.
Similarly, if $r\in C_2$, then
$\wt{E_r} = \beta \wt{F_2}$, for some $\beta \in \R$.

If follows from $\wt{F_1}\wt{F_2}=\wt{F_2}\wt{F_1}$ that the $\wt{E_r}$'s commute with each other.
\end{proof}

Perfect state transfer occurs only between strongly cospectral vertices, equivalently vertices that are  cospectral and parallel \cite{Godsil2017}.
Two vertices $a$ and $b$ are \textsl{cospectral} if and only if 
\begin{equation*}
(E_r)_{a,a} = (E_r)_{b,b}, \quad \text{for $r=0,1,\ldots, d$.}
\end{equation*}
Thus, for each $r=0,\ldots, d$, there exists real numbers $\alpha_r$ and $\beta_r$ such that
\begin{equation*}
\wt{E_r} = \begin{bmatrix} \alpha_r & \beta_r\\ \beta_r & \alpha_r\end{bmatrix}.
\end{equation*}
In this case, the $\wt{E_r}$'s commute with each other.   

Motivated by Theorem~\ref{Thm:FC} and the above necessary conditions on cospectral vertices, 
we now generalize the notion of cospectrality to fractionally cospectrality.
\begin{definition}
\label{Def:FC}
Given a connected  graph $X$ with real weights and the spectral decomposition, $A=\sum_{r=0}^d \theta_r E_r$, of $A$.  
Let $K=\{a,b\} \subset V(X)$.
We say $a$ is \textsl{fractionally cospectral} to $b$ in $X$ if
\begin{equation*}
\wt{E_j} \wt{E_h} = \wt{E_h} \wt{E_j}, 
\quad \text{for $j, h = 0,\ldots, d$.}
\end{equation*}
\end{definition}

Theorem~3.1 of \cite{Godsil2017} gives six equivalent statements to two vertices being cospectral.
We now give the analogous characterizations for fractionally cospectral vertices in a connected graph.

For a vertex $a$ in $X$, the \textsl{walk matrix} $M_a$ relative to the $a$ is defined to be
\begin{equation*}
M_a =\begin{bmatrix} e_a & Ae_a & A^2 e_a & \cdots & A^{n-1}e_a\end{bmatrix}.
\end{equation*}
Then $(M_a^T M_b)_{j,h} = e_a^T A^{j+h-2} e_b$  counts to number of walks of length $(j+h-2)$ from $a$ to $b$,
The \textsl{walk generating function from $a$ to $b$} is the formal power series
\begin{equation}
\label{Eqn:WalkGenA}
W_{a,b}(X,y) = \sum_{k\geq 0} (e_a^T A^ke_b )y^k = e_a^T(I-yA)^{-1}e_b = \sum_{r=0}^d \frac{(E_r)_{a,b}}{1-\theta_r y}.
\end{equation}
We use $\phi(X,y)$ to denote the characteristic polynomial of the adjacency matrix of $X$.  From Pages~30 and 52 of \cite{GodsilAC}, we see that
\begin{equation}
\label{Eqn:WalkGen}
W_{a,a}(X,y)=\sum_{r=0}^d \frac{(E_r)_{a,a}}{1-y\theta_r}=y^{-1}\frac{\phi(X\backslash\{a\},y^{-1})}{\phi(X,y^{-1})}.
\end{equation}

\begin{theorem}
\label{Thm:FCequiv}
Let $X$ be a connected  weighted graph and let $A=\sum_{r=0}^d \theta_rE_r$ be the spectral decomposition of $A$.
For vertices $a$  and $b$ in $X$,
the following statements are equivalent.
\begin{enumerate}[i.]
\item
\label{Cond:FCequiv1}
$a$ is fractionally cospectral to $b$.
\item
\label{Cond:FCequiv2}
There exists an orthonormal basis $v_1,\ldots, v_n$ of eigenvectors of $A$ and non-zero real numbers $p$ and $q$
such that 
either $v_j(a) = \frac{p}{q} v_j(b)$ or $v_j(a) = -\frac{q}{p} v_j(b)$, for $j=1,\ldots,n$.
\item
\label{Cond:FCequiv3}
For $r=0,1,\ldots,d$, 
\begin{equation*}
(E_r)_{a,a} - (E_r)_{b,b} = \left(\frac{p}{q} - \frac{q}{p}\right)(E_r)_{a,b}.
\end{equation*}
\item
\label{Cond:FCequiv4}
For $k\geq 0$, 
\begin{equation*}
(A^k)_{a,a} - (A^k)_{b,b} = \left(\frac{p}{q} - \frac{q}{p}\right)(A^k)_{a,b}.
\end{equation*}
\item
\label{Cond:FCequiv5}
$W_{a,a}(X,y) - W_{b,b}(X,y) = \left(\frac{p}{q} - \frac{q}{p}\right) W_{a,b}(X,y)$.
\item
\label{Cond:FCequiv6}
$M_a^TM_a -M_b^TM_b=  \left(\frac{p}{q} - \frac{q}{p}\right) M_a^T M_b$.
\item $\phi(X\backslash \{a\},y) - \phi(X\backslash \{b\}, y) =$ \\ \phantom{a} \hspace{1.5cm} $  \left(\frac{p}{q} - \frac{q}{p}\right)\sqrt{\phi(X\backslash \{a\},y)\phi(X\backslash \{b\}, y)-\phi(X,x)\phi(X\backslash\{a,b\},y)}.$
\label{Cond:FCequiv7}
\item
\label{Cond:FCequiv8}
The $\alg{A}$-modules generated by $(p e_a+q e_b)$ and $(-qe_a+pe_b)$ are orthogonal subspaces of $\C^{V(X)}$.
\end{enumerate}
In  particular, if $p=\pm q$ then $a$ and $b$ are cospectral.
\end{theorem}
\begin{proof}
Assume $a$ and $b$ are fractionally cospectral vertices in $X$.
Since $\wt{E_r}$'s are commuting and symmetric, they have spectral decomposition
\begin{equation*}
\wt{E_r} = \lambda_{r,1} M_1+\lambda_{r,2} M_2, \quad \text{for $r=0,1,\ldots,d$.}
\end{equation*}
Since $X$ is connected, at least one of the $\wt{E_r}$'s is not diagonal which implies $M_1$ and $M_2=I_2-M_1$ are non-diagonal.
There exist non-zero real numbers $p$ and $q$ such that $p^2+q^2=1$, and
\begin{equation*}
M_1= \begin{bmatrix} p \\ q \end{bmatrix}\begin{bmatrix} p & q \end{bmatrix}
\quad \text{and} \quad
M_2= \begin{bmatrix} -q \\ p \end{bmatrix}\begin{bmatrix} -q & p \end{bmatrix}.
\end{equation*}
Observe that $M_1$  commutes with $\wt{E_r}$, for $r=0,\ldots, d$, so applying Proposition~\ref{Prop:Commute} to 
$H=M_1$ gives an orthonormal basis of real eigenvectors $v_1,\ldots, v_n$ of $A$ such that either $\tp{v_j}=\0$ or $\tp{v_j}$ is an eigenvector of $M_1$,
for $j=1,\ldots,n$.    If $\tp{v_j}\neq \0$ then either $M_1 \tp{v_j} = 1\cdot \tp{v_j}$ or $M_1 \tp{v_j} = 0 \cdot \tp{v_j}$.  
The former case implies $\tp{v_j} \in \spn\{\begin{bmatrix}p&q\end{bmatrix}^T\}$ while the latter implies $\tp{v_j} \in \spn\{\begin{bmatrix}-q&p\end{bmatrix}^T\}$.
As a result, (\ref{Cond:FCequiv1}) implies (\ref{Cond:FCequiv2}).

Suppose (\ref{Cond:FCequiv2}) holds.  Then for $j=1,\ldots,n$, we have
\begin{equation*}
(\tp{v_j}\tp{v_j}^T)_{a,a}-(\tp{v_j}\tp{v_j}^T)_{b,b} = \left(\frac{p}{q}-\frac{q}{p}\right) (\tp{v_j}\tp{v_j}^T)_{a,b}.
\end{equation*}
Statement  (\ref{Cond:FCequiv3}) follows from the fact that $E_r$ is the sum of $(\tp{v_j}\tp{v_j}^T)$ over the $v_j$'s in the $\theta_r$-eigenspace of $A$.

Suppose (\ref{Cond:FCequiv3}) holds.   For $h, l = 0,\ldots, d$,   the diagonal entries of $\wt{E_h}\wt{E_l}  - \wt{E_l} \wt{E_h}$ are zero, and
its  off-diagonal entries are equal to
\begin{equation*}
(E_l)_{a,b} \Big( (E_h)_{a,a}- (E_h)_{b,b}\Big) - (E_h)_{a,b}\Big( (E_l)_{a,a}- (E_l)_{b,b} \Big) 
\end{equation*}
which is zero after applying  (\ref{Cond:FCequiv3}) to $E_h$ and $E_l$.  Therefore  (\ref{Cond:FCequiv1}) holds

Statements  (\ref{Cond:FCequiv3}) and  (\ref{Cond:FCequiv4}) are equivalent because $A^k \in \spn\{E_0,\ldots, E_d\}$, for $k\geq 0$,
and $E_r$ is a polynomial in $A$, for $r=0,\ldots,d$.

It follows from Equation~(\ref{Eqn:WalkGenA}) that (\ref{Cond:FCequiv4}) is equivalent to (\ref{Cond:FCequiv5}).   
The proof of Lemma~2.1 of \cite{Godsil2017} shows the equivalence of 
(\ref{Cond:FCequiv5}) and (\ref{Cond:FCequiv6}).   The equivalence of (\ref{Cond:FCequiv5}) and (\ref{Cond:FCequiv7}) follows from Equation~(\ref{Eqn:WalkGen})
and Corollary~4.1.3 of
\cite{GodsilAC}.

Statement (\ref{Cond:FCequiv3}) holds if and only if
\begin{equation*}
\langle E_r (p e_a+q e_b), E_r (-qe_a+pe_b)\rangle=0, \quad \text{for $r=0,\ldots,d$},
\end{equation*}
which is equivalent to (\ref{Cond:FCequiv8}).
\end{proof}

\begin{corollary}
\label{Cor:FC}
Suppose $a$ is fractionally cospectral to $b$ in a simple graph $X$ and one of the following conditions hold.
\begin{enumerate}[i.]
\item 
\label{Cond:FCi}
$a$ is adjacent to $b$.
\item
\label{Cond:FCii}
$X$ is bipartite, and the distance between $a$ and $b$ is odd.
\item
\label{Cond:FCiii}
$a$ and $b$ have the same degree and are at distance two in $X$.
\item
\label{Cond:FCiv}
$X$ is regular and connected.
\end{enumerate}
Then $a$ and $b$ are cospectral.
\end{corollary}
\begin{proof}
By Theorem~\ref{Thm:FCequiv}~(\ref{Cond:FCequiv4}), if there exists $k>0$ such that 
\begin{equation}
\label{Eqn:FC}
(A^k)_{a,a}-(A^k)_{b,b}=0 \quad \text{and} \quad (A^k)_{a,b} \neq 0,
\end{equation}
then $\frac{p}{q}-\frac{q}{p}=0$, and $a$ and $b$ are cospectral.

For Condition~(\ref{Cond:FCi}),  (\ref{Eqn:FC}) holds for $k=1$.
For Condition~(\ref{Cond:FCii}), (\ref{Eqn:FC}) holds when $k$ is the distance of $a$ and $b$.
For Condition~(\ref{Cond:FCiii}), (\ref{Eqn:FC}) holds for $k=2$.

For Condition~(\ref{Cond:FCiv}), one of the $E_r$'s is $\frac{1}{|V(X)|} J$. 
If follows from Theorem~\ref{Thm:FCequiv}~(\ref{Cond:FCequiv3}) that 
$\frac{p}{q}-\frac{q}{p}=0$, and $a$ and $b$ are cospectral.
\end{proof}
\section{Strongly fractionally cospectral vertices}
\label{Section:SFC}

In a graph $X$, two vertices $a$ and $b$ are \textsl{parallel} if the column vectors $E_r e_a$ and $E_r e_b$ are parallel, for $r=0,1,\ldots, d$.
Further, if $E_re_a = \pm E_r e_b$, for $r=0,\ldots,d$, then we say $a$ and $b$ are \textsl{strongly cospectral}.
Lemma~4.1 of \cite{Godsil2017} states that two vertices are strongly cospectral if and only if they are both cospectral and parallel.
In this section, we show that $a$ and $b$ being fractionally cospectral and parallel is equivalent to $X$ being properly decomposable with respect to $\{a,b\}$, which motivates  Definition~\ref{Def:SFC}.

\begin{theorem}
\label{Thm:SFC}
A connected weighted graph $X$ is properly decomposable with respect to $K=\{a, b\}$ if and only if 
$a$ and $b$ are fractionally cospectral and parallel.
\end{theorem}
\begin{proof}
Suppose $X$ is properly decomposable with respect to $K=\{a, b\}$. By Theorem~\ref{Thm:FC}, $a$ is fractionally cospectral to $b$, and by Theorem~\ref{thm:inequalityandparallel}, Theorem~\ref{Thm:Decomp} and Proposition~\ref{Prop:Connected}, vertices $a$ and $b$ are parallel.

Conversely, assume the $\wt{E_r}$'s commute with each other and there exists $\sigma_r \in \R$ such that
\begin{equation*}
E_r e_a = \sigma_r E_r e_b, \quad \text{for $r=0,\ldots, d$}.
\end{equation*}
Since the $\wt{E_r}$'s are simultaneously diagonalizable, there exists $2\times 2$ orthogonal projection matrices 
$M_1$ and $M_2$ such that
\begin{equation*}
\wt{E_r} = \lambda_{r,1} M_1 +\lambda_{r,2} M_2.
\end{equation*}
As $\rk(\wt{E_r}) \leq 1$, at most one of $\lambda_{r,1}$ and $\lambda_{r,2}$ is non-zero for each $r$.

Define $C_1=\{r: \lambda_{r,1}\neq 0\}$, $C_2=\{r: \lambda_{r,2}\neq 0\}$ and $C_3=\{r: \lambda_{r,1}=\lambda_{r,2}=0\}$.
For $j=1,2,3$, let 
\begin{equation*}
F_j = \sum_{r\in C_j}E_r =
\begin{bmatrix}
\wt{F_j} & B_j^T\\
B_j & C_j
\end{bmatrix}.
\end{equation*}   
There exist $\beta_1, \beta_2\in \R$ such that
\begin{equation*}
\wt{F_j} = \sum_{r\in C_j} \wt{E_r} = \beta_j M_j, \quad \text{for $j=1,2$,}
\end{equation*}
and $\wt{F_3}=\0$.

Then $\wt{F_1}+\wt{F_2}+\wt{F_3} = \wt{(\sum_{r=0}^d E_r)} = I_2 $, so
\begin{equation*}
\beta_1 M_1 + \beta_2 M_2 +\0 =I_2
\end{equation*}
which yields $\beta_1=\beta_2=1$.  Hence $\wt{F_j}^2=\wt{F_j}$, for $j=1,2,3$.

For $j=1,2,3$,  $F_j^2=F_j$ implies $B_j^T B_j = \0$.
Therefore, $B_j=\0$ and $F_j D_K = D_K F_j$.   As $X$ is connected, $\wt{F_1}$ is non-diaogonal and  $X$ is properly decomposable with respect to $\{a,b\}$.
\end{proof}

\begin{definition}
\label{Def:SFC}
We say $a$ is \textsl{strongly fractionally cospectral} to $b$ if $X$
is properly decomposable with respect to $\{a,b\}$.
\end{definition}


\begin{corollary}
\label{Cor:SFC}
Let $X$ be a connected weighted graph.
Let $K=\{a,b\} \subset V(X)$ and $P_{min}^K = \{C_1,\ldots, C_z\}$.  The following statements are equivalent.
\begin{enumerate}[i.]
\item
\label{Cond:SFC1}
The vertices $a$ and $b$ are strongly fractionally cospectral in $X$ with
\begin{equation*}
\wt{F_1} = \begin{bmatrix} p \\ q \end{bmatrix}\begin{bmatrix} p & q \end{bmatrix}
\quad \text{and} \quad
\wt{F_2} = \begin{bmatrix} -q \\ p \end{bmatrix}\begin{bmatrix} -q & p \end{bmatrix}.
\end{equation*}
for some non-zero $p$ and $q$ satisfying $p^2+q^2=1$.  
\item
\label{Cond:SFC2}
For $r=0,\ldots,d$,   
\begin{equation*}
E_re_a=E_re_b=\0 \quad \text{if $r\not\in C_1 \cup C_2$}
\end{equation*} 
and
\begin{equation*}
E_r e_a = 
\begin{cases}
\frac{p}{q} E_r e_b & \text{if $r\in C_1$,}\\
\frac{-q}{p} E_r e_b & \text{if $r\in C_2$.}
\end{cases}
\end{equation*}
\item
\label{Cond:SFC3}
For any eigenvector $v$ of $A$, either $v(a) = \frac{p}{q} v(b)$ or $v(a) = -\frac{q}{p} v(b)$.  
\end{enumerate}
\end{corollary}

We have $p=\pm q$ when $a$ is strongly cospectral to $b$.   In this case, 
\begin{equation*}
Q=\sum_{r\in C_1} E_r - \sum_{s\in C_2}E_s
\end{equation*}
is the orthogonal symmetry of $X$ mentioned in Theorem~11.2 of \cite{Godsil2017}.

Lemma~8.3 of \cite{Godsil2017} gives a characterization of parallel vertices in terms of characteristic polynomials of the graphs $X$ and $X\backslash\{a,b\}$.
We use this result to generalize the characterization of strongly cospectral vetices in Lemma~8.4 of \cite{Godsil2017} to strongly fractional cospectrality.


\begin{corollary}
\label{Cor:SFC-poles}
Distinct vertices $a$ and $b$ in $X$ are strongly fractionally cospectral if and only if they are fractionally cospectral and all poles of 
\begin{equation*}
\frac{\phi(X\backslash\{a,b\},y)}{\phi(X,y)}
\end{equation*}
are simple.
\end{corollary}

\section{Fractional revival on two vertices}
\label{Section:FR2}


In this section, we assume $X$ is a connected  graph with integer weights.  
Suppose proper fractional revival occurs from $a$ to $b$ in $X$ at time $\tau>0$,
that is,
\begin{equation}
\label{Eqn:ProperFR1}
U(\tau) e_a = \alpha e_a+ \beta e_b,
\end{equation}
for some non-zero complex numbers $\alpha$ and $\beta$ satisfying $|\alpha|^2+|\beta|^2=1$.

Let $P_{min}^K = \{C_1, C_2, \ldots, C_z\}$.
Since $a$ and $b$ are strongly fractionally cospectral, there exists non-zero real numbers $p$ and $q$  such that $p^2+q^2=1$ and, without loss of generality, 
\begin{equation*}
\wt{F_1} = \begin{bmatrix} p \\ q \end{bmatrix}\begin{bmatrix} p & q \end{bmatrix}
\quad \text{and} \quad
\wt{F_2} = \begin{bmatrix} -q \\ p \end{bmatrix}\begin{bmatrix} -q & p \end{bmatrix}.
\end{equation*}
We have $\wt{F_j}=\0$, for $j =3,\ldots,z$.  From the proof of Proposition~\ref{Prop:ZeroBlock}, we see that $|C_j|=1$, for $j =3,\ldots,z$.

Applying Corollary~\ref{Cor:SFC}, Equation~(\ref{Eqn:ProperFR1}) holds if and only if
\begin{equation}
\label{Eqn:ProperFR2}
e^{-\ii\tau \theta_r} =
\begin{cases}
\alpha + \frac{q}{p}\beta  & \text{if $r\in C_1$,}\\
\alpha - \frac{p}{q} \beta & \text{if $r\in C_2$}
\end{cases}
\end{equation}
which is equivalent to
\begin{equation*}
\tau(\theta_s-\theta_r) =
\begin{cases}
0 \md{2\pi} & \text{if $r,s \in C_1$ or $r,s \in C_2$,}\\
\mu \md{2\pi} & \text{if $r\in C_1$ and $s\in C_2$,}
\end{cases}
\end{equation*}
for some real number $\mu$ satisfying
$e^{\ii \mu} = 
\left(\alpha + \frac{q}{p}\beta\right)\left(\alpha - \frac{q}{p}\beta\right)^{-1}$.
Note that 
\begin{equation*}
\left|\alpha +  \frac{q}{p}\beta  \right|=\left|\alpha -  \frac{p}{q}\beta  \right|=1
\end{equation*}
is equivalent to 
\begin{equation*}
\frac{\alpha}{\beta} + \frac{\cj{\alpha}}{\cj{\beta}} = \frac{p}{q}-\frac{q}{p}.
\end{equation*}

The results in \cite{ChanCoutinhoTamonVinetZhan2} and Section~5 of \cite{ChanCoutinhoTamonVinetZhan} are restricted to fractional revival between strongly cospectral vertices. We are now ready to extend them to the general case.

By Corollary~\ref{Cor:K-FRratio}, there exist a square-free integer $\Delta$ and reals $\rho_1$, $\rho_2$ such that
\begin{eqnarray*}
\theta_h &=& \rho_1 + \sigma_h \sqrt{\Delta}, \quad \text{for $h\in C_1$,} \quad \text{and}\\
\theta_j &=& \rho_2 + \omega_j \sqrt{\Delta}, \quad \text{for $j\in C_2$,} 
\end{eqnarray*}
where $\sigma_h$'s and $\omega_j$'s are reals satisfying
\begin{equation*}
\sigma_{h} - \sigma_{h'} = \frac{\theta_h-\theta_{h'}}{\sqrt{\Delta}} \in \Z 
\quad \text{and}\quad
\omega_{j} - \omega_{j'} = \frac{\theta_j-\theta_{j'}}{\sqrt{\Delta}} \in \Z ,
\end{equation*}
for all $h, h' \in C_1$ and $j, j' \in C_2$.

Let
\begin{equation*}
g = \gcd \left\{\frac{\theta_r-\theta_s}{\sqrt{\Delta}} : r,s \in C_1 \quad\text{or}\quad  r,s \in C_2 \right\}.
\end{equation*}
Then, for $h, h'\in C_1$ and $j, j'\in C_2$,
\begin{equation}
\label{Eqn:C1C2}
\tau(\theta_{h}-\theta_{h'}) = \tau(\theta_{j}-\theta_{j'})=0 \md{2\pi}
\end{equation} 
if and only if 
\begin{equation*}
\tau = \frac{2\pi k}{g\sqrt{\Delta}}, \quad\text{for some integer $k$.}
\end{equation*}
Hence the only times fractional revival between $a$ and $b$ can occur are some integer multiples of $ \frac{2\pi}{g\sqrt{\Delta}}$.

\begin{proposition}
\label{Prop:FRgcd}
Let $a$ and $b$ be strongly fractionally cospectral vertices in a connected graph $X$ with integer weights.   
Let  $g, \Delta$, $\rho_i$'s, $C_i$'s, $\sigma_h$'s and $\omega_j$'s be defined above.
Proper fractional revival between $a$ and $b$ occurs at time $\tau= \frac{2\pi k}{g\sqrt{\Delta}}$ if and only if there exist $h\in C_1$ and $j\in C_2$ 
such that
\begin{equation*}
\frac{k}{g\sqrt{\Delta}} \Big(\rho_2 - \rho_1 +(\omega_j-\sigma_h)\sqrt{\Delta} \Big) \not \in \Z.
\end{equation*}
\end{proposition}
\begin{proof}
It follows from Equation~(\ref{Eqn:C1C2}) that, for all $h, h'\in C_1$ and $j, j'\in C_2$
\begin{equation*}
e^{-\ii\tau h} = e^{-\ii\tau h'}
\quad \text{and}\quad
e^{-\ii\tau j} = e^{-\ii\tau j'}.
\end{equation*}
Equation~(\ref{Eqn:ProperFR2}) holds if and only if
\begin{equation*}
\alpha = \frac{e^{-\ii \tau \theta_h}+e^{-\ii \tau \theta_j}}{2}
\quad \text{and}\quad
\beta = \frac{q}{p}\frac{(e^{-\ii \tau \theta_h}-e^{-\ii \tau \theta_j})}{2},
\end{equation*}
for some $h\in C_1$ and $j\in C_2$.
Then
$\beta \neq 0$ if and only if $e^{\ii\tau(\theta_j-\theta_h)} \neq 1$.
The latter holds if and only if 
\begin{equation*}
\frac{k}{g\sqrt{\Delta}} \Big(\rho_2 - \rho_1 +(\omega_j-\sigma_h)\sqrt{\Delta} \Big) \not \in \Z.
\end{equation*}
In this case, we have $U(\tau)e_a = \alpha e_a +\beta e_b$ with $\beta \neq 0$.
\end{proof}

\begin{theorem}
\label{Thm:FR}
Let $X$ be a connected graph with integer weights and $a,b \in V(X)$.  Proper fractional revival between $a$ and $b$ occurs in $X$ if and only if
the following conditions hold.
\begin{enumerate}[i.]
\item
\label{Cond:FR1}
$a$ and $b$ are strongly fractionally cospectral vertices with $C_1$ and $C_2$ defined as above.
\item
\label{Cond:FR2}
There exist a square-free integer $\Delta$ and real numbers $\rho_1$, $\rho_2$ such that
\begin{eqnarray*}
\theta_h &=& \rho_1 + \sigma_h \sqrt{\Delta}, \quad \text{for $h\in C_1$,} \quad \text{and}\\
\theta_j &=& \rho_2 + \omega_j \sqrt{\Delta}, \quad \text{for $j\in C_2$,} 
\end{eqnarray*}
where $\sigma_h$'s and $\omega_j$'s are real numbers satisfying
\begin{equation*}
\sigma_{h} - \sigma_{h'}  \in \Z 
\quad \text{and}\quad
\omega_{j} - \omega_{j'}  \in \Z ,
\end{equation*}
for all $h, h' \in C_1$ and $j, j' \in C_2$.
\item
\label{Cond:FR3}
Let $g = \gcd \{\frac{\theta_r-\theta_s}{\sqrt{\Delta}} : r,s \in C_1 \quad\text{or}\quad  r,s \in C_2 \}$.
There exist $h\in C_1$ and $j\in C_2$ such that
\begin{equation*}
\frac{k}{g\sqrt{\Delta}} \Big(\rho_2 - \rho_1 +(\omega_j-\sigma_h)\sqrt{\Delta} \Big) \not \in \Z.
\end{equation*}
\end{enumerate}
Moreover, if these conditions hold,  $\frac{2\pi}{g\sqrt{\Delta}}$ is the minimum time proper fractional revival between $a$ and $b$ occurs in $X$.
\end{theorem}

\section{Polygamy}
\label{Section:Polygamy}

Perfect state transfer has the monogamous property in that a vertex cannot have perfect state transfer to two distinct vertices.
In \cite{ChanCoutinhoTamonVinetZhan}, the question of whether fractional revival exhibits the same monogamous property is raised.
We answer this question in the negative by constructing weighted graphs that have  fractional revival between every pair of vertices.

We start with a normalized Hadamard matrix $H$ of order $n$, for some $n\geq 4$.  
We first construct the Laplacian matrix of a connected weighted graph that has fractional revival from $v_0$ to $v_k$, for $k=1,\ldots, n-1$.
Let $p_1, \ldots, p_{n-1}$ be distinct odd primes.
For $r=1,\ldots, n-1$, we use Chinese remainder theorem to find the unique solution, $\lambda_r$,  in $[0, \Pi_{j=1}^{n-1} p_j]$ for the following system of equations.
\begin{equation}
\label{Eqn:CRT}
\begin{cases}
x=& (1-H_{1,r})/2 \md{p_1}\\
x=& (1-H_{2,r})/2 \md{p_2}\\
\vdots & \vdots\\
x=& (1-H_{n-1,r})/2\md{p_{n-1}}\\
\end{cases}
\end{equation}
Let $\lambda_0=0$.
For $r=0,1,\ldots, n-1$, let $S_r =\{j : H_{j,r}=1\}$.
Then 
\begin{equation*}
\lambda_r = 0 \md{\ \Pi_{j\in S_r} p_j}
\quad \text{and}\quad
\lambda_r  = 1 \md{\ \Pi_{h \not \in S_r} p_h}.
\end{equation*}
Since the columns of $H$ are distinct, so are the sets $S_0, \ldots, S_{n-1}$.   We conclude that $\lambda_0, \lambda_1, \ldots, \lambda_{n-1}$ are distinct integers.

Define $D$ to be the $n\times n$ diagonal matrix where $D_{r,r}=\lambda_r$, for $r=0,1,\ldots, n-1$, and
\begin{equation*}
L = \frac{1}{n} HD H^T.
\end{equation*}
Then $L$ is the Laplacian matrix of a Hadamard diagonalizable graph with $n$ distinct eigenvalues.
Hence the principal idempotents of $L$ are
\begin{equation*}
E_r= \frac{1}{n}(He_r) (H e_r)^T, \quad \text{for $r=0,\ldots, n$},
\end{equation*}
and 
\begin{equation*}
U(t) =e^{-\ii t L}= \sum_{r=0}^{n-1} e^{-it\lambda_r} E_r.
\end{equation*}
As $H$ is normalized, $E_r e_0 = \frac{1}{n} H e_r$, and
\begin{equation*}
U(t) e_0 = \frac{1}{n} H \left(\sum_{r=0}^{n-1} e^{-it\lambda_r} e_r\right).
\end{equation*}
It follows from Equations~(\ref{Eqn:CRT}) that
\begin{equation*}
e^{-\frac{2\pi\lambda_r}{p_k} \ii} =
\begin{cases}
1 & \text{if $H_{k,r}=1$,}\\
e^{-\frac{2\pi}{p_k}\ii} & \text{if $H_{k,r}=-1$.}
\end{cases}
\end{equation*}
Thus
\begin{eqnarray*}
U(\frac{2\pi}{p_k})e_0
&=& \frac{1}{n} H \left((1) \frac{1}{2} (H^Te_0+H^Te_k) + (e^{-\frac{2\pi}{p_k}\ii} ) \frac{1}{2} (H^Te_0-H^Te_k)\right)\\
&=& \frac{1}{2}\left(1+e^{-\frac{2\pi}{p_k}\ii}\right) e_0 + \frac{1}{2}\left(1-e^{-\frac{2\pi}{p_k}\ii}\right) e_k
\end{eqnarray*}
and fractional revival occurs from vertex $v_0$ to vertex $v_k$ at time $\frac{2\pi}{p_k}$, for $k=1,\ldots,p_{n-1}$.

When $H$ is the character table of $\Z_2^m$, that is,
\begin{equation*}
H =\begin{bmatrix} 1&1\\1&-1 \end{bmatrix}^{\otimes m},
\end{equation*}
$L$ is the Laplacian matrix of a cubelike graph with vertex set $\Z_2^m$. 
As a cubelike graph is vertex transitive, fractional revival occurs between any pair of vertices having the same difference as $v_0-v_k$ at time
$\frac{2\pi}{p_k}$.

It follows from Theorem~5 of \cite{BarikFallatKirkland} that $X$ is regular if its Laplacian matrix is diagonalizable by a Hadamard matrix.
We conclude that the adjacency matrix of the cubelike graph constructed above also admits fractional revival between every pair of vertices.

The outstanding question here is whether there exists a simple graph in which fractional revival is polygamous.

\begin{example}
\label{Ex:H4}
Let
\begin{equation*}
H=\begin{bmatrix} 1&1\\1&-1\end{bmatrix}^{\otimes 2}=\begin{bmatrix} 1&1&1&1\\1&-1&1&-1\\1&1&-1&-1\\1&-1&-1&1\end{bmatrix}.
\end{equation*}
The rows and columns of $H$ are indexed by vertices
$v_0=(0,0)$, $v_1=(1,0)$, $v_2=(0,1)$ and $v_3=(1,1)$, which are elements of $\Z_2^2$.

Let $p_1=3$, $p_2=5$, and $p_3=7$.  Solving Equations~(\ref{Eqn:CRT}) gives $\lambda_1=85$, $\lambda_2=36$, $\lambda_3=91$.
Let
\begin{equation*}
L = \frac{1}{4}H \begin{bmatrix}0&0&0&0\\0&85&0&0\\0&0&36&0\\0&0&0&91\end{bmatrix}H^T
= \begin{bmatrix} 53 & -35&-10.5 & -7.5\\-35&53&-7.5&-10.5\\ -10.5 &-7.5&53&-35\\-7.5&-10.5&-35&53\end{bmatrix}.
\end{equation*}
Then
\begin{equation*}
U(\frac{2\pi}{3}) = 
\frac{1}{2} 
\begin{bmatrix}
1+e^{-\frac{2\pi}{3}\ii}& 1-e^{-\frac{2\pi}{3}\ii}&0&0\\
1-e^{-\frac{2\pi}{3}\ii}& 1+e^{-\frac{2\pi}{3}\ii} & 0&0\\
0&0&1+e^{-\frac{2\pi}{3}\ii}&1-e^{-\frac{2\pi}{3}\ii}\\
0&0&1-e^{-\frac{2\pi}{3}\ii}&1+e^{-\frac{2\pi}{3}\ii}
\end{bmatrix},
\end{equation*}
\begin{equation*}
U(\frac{2\pi}{5}) = 
\frac{1}{2} 
\begin{bmatrix}
1+e^{-\frac{2\pi}{5}\ii} & 0&1-e^{-\frac{2\pi}{5}\ii}&0\\
0& 1+e^{-\frac{2\pi}{5}\ii} & 0&1-e^{-\frac{2\pi}{5}\ii}\\
1-e^{-\frac{2\pi}{5}\ii}&0&1+e^{-\frac{2\pi}{5}\ii}&0\\
0&1-e^{-\frac{2\pi}{5}\ii}&0&1+e^{-\frac{2\pi}{5}\ii}
\end{bmatrix},
\end{equation*}
and
\begin{equation*}
U(\frac{2\pi}{7}) = 
\frac{1}{2} 
\begin{bmatrix}
1+e^{-\frac{2\pi}{7}\ii} & 0&0&1-e^{-\frac{2\pi}{7}\ii}\\
0& 1+e^{-\frac{2\pi}{7}\ii} & 1-e^{-\frac{2\pi}{7}\ii}&0\\
0&1-e^{-\frac{2\pi}{7}\ii}&1+e^{-\frac{2\pi}{7}\ii}&0\\
1-e^{-\frac{2\pi}{7}\ii}&0&0&1+e^{-\frac{2\pi}{7}\ii}
\end{bmatrix}.
\end{equation*}
We see that fractional revival occurs between every pair of vertices.
\end{example}

\section{Prescribed fractional revival}

We are now ready to construct a weighted graph $X$ on $n\geq 4$ vertices that has (proper) fractional revival between $a$ and $b$ at a given time $\tau=2\pi\xi$
and 
\begin{equation*}
\wt{U(\tau)} = H,
\end{equation*}
for any given $2\times 2$ non-diagonal symmetry unitary matrix $H$.


Suppose $H$ has the spectral decomposition
\begin{equation*}
H = \lambda_1 \begin{bmatrix} p\\q\end{bmatrix}\begin{bmatrix} p& q\end{bmatrix} + \lambda_2 \begin{bmatrix}-q\\p\end{bmatrix}\begin{bmatrix} -q& p\end{bmatrix},
\end{equation*}
for some non-zero real numbers $p$ and $q$ satisfying $p^2+q^2=1$, and for some distinct $\lambda_1, \lambda_2\in \C$.   

We first choose the eigenvalues of $A$.   Let $\theta_1$ and $\theta_3$ be real numbers satisfying 
\begin{equation*}
e^{-\ii \tau \theta_1} = \lambda_1
\quad \text{and}\quad
e^{-\ii \tau \theta_3} = \lambda_2.
\end{equation*}
Then we pick non-zero integers $\sigma$ and $\omega$ to define
\begin{equation}
\label{Eqn:Prescribed}
\theta_2= \theta_1+ \frac{\sigma}{\xi}
\quad \text{and}\quad
\theta_4= \theta_3+ \frac{\omega}{\xi}
\end{equation}
so that $\theta_1,\theta_2, \theta_3$, and $\theta_4$ are all distinct.
Finally, pick any $\theta_5, \ldots, \theta_n\in \R$ that are distinct from $\theta_1, \ldots, \theta_4$.

We now choose an orthonormal basis of real eigenvectors for $A$.  Let $u_1,\ldots, u_{n-2}$ be an orthonormal basis of real vectors in $\C^{n-2}$.
Define the $n\times n$ orthogonal matrix
\begin{equation*}
P=
\begin{bmatrix}
\frac{p}{\sqrt{2}} & \frac{-p}{\sqrt{2}} & \frac{-q}{\sqrt{2}} & \frac{q}{\sqrt{2}} & 0 &0& \ldots & 0\\
\frac{q}{\sqrt{2}} & \frac{-q}{\sqrt{2}} & \frac{p}{\sqrt{2}} & \frac{-p}{\sqrt{2}} & 0 & 0&\ldots & 0\\
\frac{1}{\sqrt{2}} u_1 & \frac{1}{\sqrt{2}} u_1 &\frac{1}{\sqrt{2}} u_2 &\frac{1}{\sqrt{2}} u_2 & u_3 & u_4 & \ldots & u_{n-2}
\end{bmatrix}.
\end{equation*}
We define the adjacency matrix of $A$ to be
\begin{equation*}
A = P D P^T,
\end{equation*}
where $D$ is  the diagonal matrix with $D_{j,j}=\theta_j$, for $j=1,\ldots,n$. 

For $j=1,\ldots, 4$, $\theta_j$ is distinct from all other eigenvalues of $A$, so
\begin{equation*}
E_j =  Pe_j (Pe_j)^T
\end{equation*}
is a principal idempotent of $A$.   It is straightforward to show that in $P_{min}^{\{a,b\}}$
\begin{equation*}
C_1 = \{1,2\}
\quad \text{and} \quad
C_2 = \{3,4\}.
\end{equation*}
We see that Condition~(\ref{Cond:SFC2}) in Corollary~\ref{Cor:SFC} holds, and conclude that $a$ and $b$ are strongly fractionally cospectral in $X$.
The ratio conditions with respect to $P_{min}^K$ follows from
Equation~(\ref{Eqn:Prescribed}).
By Theorem~\ref{Thm:Char2}, $X$ admits proper fractional revival from $a$ to $b$ in $X$.  Moreover 
\begin{equation*}
\wt{U(\tau)} = e^{-\ii\tau \theta_1} \begin{bmatrix} p^2 & pq\\pq&q^2\end{bmatrix} +e^{-\ii\tau \theta_3} \begin{bmatrix} q^2 & -pq\\-pq&p^2\end{bmatrix}  = H
\end{equation*}
and
\begin{equation*}
U(\tau) = \begin{bmatrix} H &\0 \\ \0 & H'\end{bmatrix},
\end{equation*}
for some $(n-2)\times (n-2)$ matrix $H'$.

\begin{example}
\label{Ex:Prescribed}
We pick $\theta_1, \ldots, \theta_4$ as defined above such that none of them is zero.
Let $\theta_5=\ldots =\theta_n=0$.   
For some $1<m<n-3$,
let
\begin{equation*}
u_1 = \begin{bmatrix}\frac{1}{\sqrt{m}} \1_m \\ \0_{n-2-m} \end{bmatrix}
\quad \text{and} \quad
u_2=\begin{bmatrix}  \0_m \\ \frac{1}{\sqrt{n-2-m}} \1_{n-2-m} \end{bmatrix}
\end{equation*}
and extend from $\{u_1, u_2\}$ to an orthonormal basis $\{u_1,\ldots,u_{n-2}\}$ in  $\C^{n-2}$.
Following the above construction, the graph $X$ with adjacency matrix
\begin{equation*}
A =
\begin{bmatrix}
\frac{(\theta_1+\theta_2)p^2+(\theta_3+\theta_4)q^2}{2} & \frac{(\theta_1+\theta_2-\theta_3-\theta_4)pq}{2} 
 &\frac{p(\theta_1-\theta_2)}{2\sqrt{m}}\1_{m}^T&\frac{-q(\theta_3-\theta_4)}{2\sqrt{n-2-m}}\1_{n-2-m}^T \\
\frac{(\theta_1+\theta_2-\theta_3-\theta_4)pq}{2} &\frac{(\theta_1+\theta_2)q^2+(\theta_3+\theta_4)p^2}{2} 
&\frac{q(\theta_1-\theta_2)}{2\sqrt{m}}\1_{m}^T&\frac{p(\theta_3-\theta_4)}{2\sqrt{n-2-m}}\1_{n-2-m}^T \\
\frac{p(\theta_1-\theta_2)}{2\sqrt{m}}\1_{m}&\frac{q(\theta_1-\theta_2)}{2\sqrt{m}}\1_{m} & \frac{\theta_1+\theta_2}{2m} J_m & \0\\
\frac{-q(\theta_3-\theta_4)}{2\sqrt{n-2-m}}\1_{n-2-m}&\frac{p(\theta_3-\theta_4)}{2\sqrt{n-2-m}}\1_{n-2-m} & \0 & \frac{\theta_3+\theta_4}{2(n-2-m)} J_{n-2-m} \\
\end{bmatrix}
\end{equation*}
has fractional revival between $a$ and $b$ at time $\tau$ with $\wt{U(\tau)}=H$.

\begin{center}
\begin{tikzpicture}
\fill (0,0) circle (1.5pt);
\draw (0,0) to [out=120, in = 180] (0,0.75) to [out=0,in=60] (0,0);
\draw (0,0) node[anchor=east]{\small $a$};
\draw (0,0)--(4,0);
\fill (4,0) circle (1.5pt);
\draw (4,0) to [out=120, in = 180] (4,0.75) to [out=0,in=60] (4,0);
\draw (4,0) node[anchor=west]{\small $b$};

\fill (-0.75,-1.5) circle (1.5pt);
\draw (-0.75,-1.5) to [out=240, in = 180] (-0.75,-2.25) to [out=0,in=300] (-0.75,-1.5);
\draw (0,-1.75) node[anchor=south]{$\cdots \cdots$};
\fill (0.75,-1.5) circle (1.5pt);
\draw (0.75,-1.5) to [out=240, in = 180] (0.75,-2.25) to [out=0,in=300] (0.75,-1.5);
\draw[dashed] (0,-1.5) ellipse (1.5 and 0.5);
\draw (-2,-1.75) node[anchor=south]{\small $K_m$};
\draw (0,0)--(-0.75,-1.5);
\draw (0,0)--(0.75,-1.5);
\draw [gray, dotted] (0,0)--(-0.25,-1.5);
\draw [gray, dotted] (0,0)--(0.25,-1.5);
\draw (0,0)--(3,-1.5);
\draw (0,0)--(5,-1.5);
\draw [gray, dotted] (0,0)--(3.5,-1.5);
\draw [gray, dotted] (0,0)--(4,-1.5);
\draw [gray, dotted] (0,0)--(4.5,-1.5);

\fill (3,-1.5) circle (1.5pt);
\draw (3,-1.5) to [out=240, in = 180] (3,-2.25) to [out=0,in=300] (3,-1.5);
\draw (4,-1.75) node[anchor=south]{$\cdots \cdots$};
\fill (5,-1.5) circle (1.5pt);
\draw (5,-1.5) to [out=240, in = 180] (5,-2.25) to [out=0,in=300] (5,-1.5);
\draw[dashed] (4,-1.5) ellipse (2 and 0.5);
\draw (6.75,-1.75) node[anchor=south]{\small $K_{n-2-m}$};
\draw (4,0)--(3,-1.5);
\draw (4,0)--(5,-1.5);
\draw [gray, dotted] (4,0)--(3.5,-1.5);
\draw [gray, dotted] (4,0)--(4,-1.5);
\draw [gray, dotted] (4,0)--(4.5,-1.5);
\draw (4,0)--(-0.75,-1.5);
\draw (4,0)--(0.75,-1.5);
\draw [gray, dotted] (4,0)--(-0.25,-1.5);
\draw [gray, dotted] (4,0)--(0.25,-1.5);

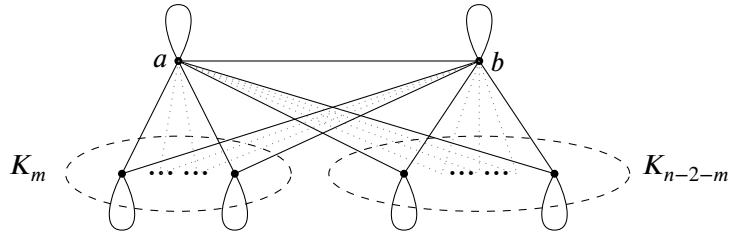
\draw (2,-2.5) node[anchor=north]{\small The underlying graph of $X$ in Example~\ref{Ex:Prescribed}};
\end{tikzpicture}
\end{center}
\end{example}

This method can be extended to $K$-fractional revival for $|K| \geq 2$.
Given a time $\tau$ and a $|K| \times |K|$ non-diagonal symmetric unitary matrix $H$,
we can always construct a weighted graph $X$ to have proper $K$-fractional revival at $\tau$ with the prescribed matrix $H$.  It is only when we wish to impose additional structure on $A$, such as having integer weights or being sparse, that difficulty arises.

\paragraph{Acknowledgements} G.L. was supported by NSF/DMS-1800738 and the Simons Foundation Collaboration Grant. O.E was supported by the Herchel Smith Harvard Undergraduate Research Program. 

We all acknowledge the support of Chris Godsil's NSERC Accelerator Grant that funded the workshop Algebraic Graph Theory and Quantum Walks, where we all started to work on this project.



\end{document}